\documentclass[reqno]{amsart}
\usepackage{amsmath}
\usepackage{bm}
\usepackage{paralist}
\usepackage{graphics} 
\usepackage{epsfig} 
\usepackage{graphicx} 
\usepackage[margin=1in]{geometry}
\usepackage{epstopdf}
 \usepackage[colorlinks=true]{hyperref}
 \usepackage{wrapfig}
 \usepackage{systeme}
 \usepackage{caption}
\usepackage{subcaption}
\usepackage{amssymb}
\usepackage{seqsplit}

\hypersetup{urlcolor=blue, citecolor=red}
\usepackage{hyperref}
\newtheorem{theorem}{Theorem}[section]
\newtheorem{corollary}[theorem]{Corollary}

\newtheorem{lemma}[theorem]{Lemma}
\newtheorem{proposition}[theorem]{Proposition}
\newtheorem{definition}[theorem]{Definition}
\newtheorem{remark}[theorem]{Remark}

\newcommand{\dd}{\mathrm{d}}
\numberwithin{equation}{section}

\begin{document}

\title[DG in time discretization for 2nd order hyperbolic PDEs] 
      {A high-order discontinuous Galerkin in time discretization for second-order hyperbolic equations }

\author{Aili Shao}
\address{Department of Mathematics, University of Oxford, Radcliffe Observatory Quarter, Woodstock Road, Oxford,UK OX2 6GG}
\email{aili.shao@maths.ox.ac.uk}

\begin{abstract}
The aim of this paper is to apply a high-order discontinuous-in-time scheme to second-order hyperbolic partial differential equations (PDEs). We first discretize the PDEs in time while keeping the spatial differential operators undiscretized. The well-posedness of this semi-discrete scheme is analyzed and \emph{a priori} error estimates are derived in the energy norm. We then combine this $hp$-version discontinuous Galerkin method for temporal discretization with an $H^1$-conforming finite element approximation for the spatial variables to construct a fully discrete scheme. \emph{A priori} error estimates are derived both in the energy norm and the $L^2$-norm. Numerical experiments are presented to verify the theoretical results.
\end{abstract}

\maketitle


\section{Introduction}
Discontinuous Galerkin (DG) methods \cite{RH, LR} have been widely and successfully used for the numerical approximation of partial differential equations (PDEs). They have been first introduced by Reed and Hill \cite{RH} to solve the hyperbolic neutron transport equations, and then generalized to elliptic and parabolic problems by Babu\v{s}ka $\&$ Zl\'{a}mal \cite{BZ}, Baker \cite{Ba}, Wheeler \cite{Wh}, Arnold \cite{Ar} and Rivi\`{e}re \cite{Riv} etc. Relevant analysis and applications of DG methods to first-order hyperbolic problems can be found in \cite{RH, LR, JP, CS2, FR, Ri, CS, SHS}. Several discontinuous Galerkin finite element methods (DGFEM) for solving wave-type equations have appeared in the literature \cite{Jo,AT, GSS,RS}.


For time-dependent problems we need to consider a suitable time integration scheme. Traditional approaches for numerical integration include explicit and implicit finite difference, Runge--Kutta  \cite{Ru, Ku} and Newmark--beta \cite{Ne} methods. Though implicit schemes are typically unconditionally stable, explicit schemes are usually preferred in engineering and physical applications because of their computational convenience. The main drawback of explicit methods is the limitation on the time step imposed by the Courant--Friedrichs--Lewy (CFL) \cite{CFL} condition. To alleviate this limitation, a time-stepping scheme based on the DGFEM was introduced. In contrast with traditional finite difference time integration schemes, for which the solution at the current time step depends on the previous steps, this discontinuous-in-time scheme on the time interval $(t_n, t_{n+1}]$ only depends on the solution at $t_{n}^{-}.$ Since the local polynomial degree is free to vary between time steps, this method is also naturally suited for an adaptive choice of the time discretization parameter. Thus, the DG method for temporal discretization has been a popular numerical integration scheme for time-dependent problems. For instance, Jamet \cite{Ja} and Eriksson \emph{et al.} \cite{EJT} discretized parabolic equations using this discontinuous-in-time scheme and derived error estimates of order $O(k^q+h^r)$ in the $L^2$ norm where $k$ and $h$ are the temporal and spatial discretization parameters, respectively, and $q$ and $r$ are the corresponding polynomial degrees. Following this approach for parabolic problems, Johnson \cite{Jo} and Saedpanah \cite{RS} applied the DG temporal discretization to the wave equation by converting the second-order hyperbolic equation into a first-order PDE system. Sch\"{o}tzau and Schwab \cite{SS} analyzed this discontinuous-in-time scheme for linear parabolic problems in the $hp$--version context; Antonietti \emph{et al.} \cite{AMSQ} applied this $hp$--DGFEM directly to systems of second-order ordinary differential equations (ODEs). 

In this paper, we generalize this high-order discontinuous-in-time method \cite{AMSQ} to second-order hyperbolic-type PDEs, which arise in a wide range of relevant applications, including acoustic, elastic, and electromagnetic wave propagation phenomena. We first discretize the PDEs in time while keeping the spatial differential operators undiscretized. The resulting weak formulation is based on weakly imposing the continuity of the approximate solution and its time derivative between time steps by penalizing jumps in these quantities in the definition of the numerical method. This discontinuous-in-time scheme used in the approximation is implicit and unconditionally stable. We then fully discretize the problem using an $H^1$-conforming finite element approximation for the spatial variables.

The paper is structured as follows. The next section is devoted to the construction of the semi-discrete DG scheme for second-order hyperbolic PDEs. We start with the statement of the problem and then discretize the problem in time.  The well-posedness of the resulting semi-discrete problem is analyzed. We also construct an appropriate energy norm arising from the variational formulation, which will be used to study the convergence of this numerical scheme. In Section \ref{semi_discrete}, we first analyze the stability of this method in the associated mesh-dependent energy norm. Then we give detailed proofs of our \emph{a priori} error estimates using the $L^2$-projection operator and the properties of Legendre polynomials. In Section \ref{fullydiscrete}, we construct a fully discrete scheme with an $H^1(\Omega)$-conforming finite element approximation in the spatial variables and carry out the convergence analysis of this fully discrete scheme. Numerical examples are provided in Section \ref{numerical} to verify the convergence results presented in Section \ref{fullydiscrete}.

\section{Semi-discrete Discontinuous Galerkin Approximation}\label{semi_discrete}
In this section, we introduce a high-order discontinuous Galerkin method of lines approach for second-order hyperbolic PDEs. We will show that the resulting variational formulation is well-posed and thus admits a unique solution in the linear space, which we will define in this section.
\subsection{A model problem}
Let $\Omega\subset \mathbb{R}^d$ be an open and bounded domain with Lipschitz boundary $\partial\Omega$. For $T>0$, we consider the following initial boundary value problem:
\begin{equation}\label{setup}
\ddot{u}(x,t)+\dot{u}(x,t)-\Delta u(x,t)=f(x,t) \mbox{ in } \Omega\times (0, T],
\end{equation}
where $f$ lies in $L^2((0,T];L^2(\Omega))$, and
\begin{equation}\label{ic}
u(x,0)=u_0(x) \in H_0^1(\Omega), \: \dot{u}(x,0)=u_1(x)\in L^2(\Omega),
\end{equation}
\begin{equation}\label{bc}
u(x,t)=0  \mbox{ on } \partial \Omega\times (0, T],
\end{equation}
are prescribed initial and boundary conditions. Here the dots over $u$ denote differentiation with respect to time $t$. We need to find a numerical approximation to the weak solution of the hyperbolic problem (\ref{setup})--(\ref{bc}). 

Before we formally define the \emph{weak solution} of the initial boundary value problem (\ref{setup})--(\ref{bc}), let us fix the notation here first. Since $\|\cdot\|_{H^1(\Omega)}$ and the seminorm $|\cdot|_{H^1(\Omega)}$ are equivalent in $H_0^1(\Omega)$ by the Poincar\'{e}--Friedrichs inequality, we define $\|v\|_{H^1(\Omega)}:=\|\nabla v\|_{L^2(\Omega)}$ for $v\in H_0^1(\Omega)$.
We denote by $\left(\cdot,\cdot\right)_{L^2}$ the inner product in $L^2(\Omega)$ and 
let $\left\langle \cdot, \cdot \right\rangle$ be the usual dual pairing between $H^{-1}(\Omega)$ and $H_0^1(\Omega)$.  Finally, we define the time-dependent bilinear form 
$$B[u,\varphi; t]:=\left(\dot{u}(t),\varphi\right)_{L^2}+\left(\nabla u(t),\nabla\varphi\right)_{L^2}$$  for $u\in L^2((0,T]; H_0^1(\Omega))$ with $\dot{u}\in L^2((0,T];L^2(\Omega))$ and $\varphi\in H_0^1(\Omega)$.
\begin{definition}\label{weaksol}
We say that a function 
$u\in L^2((0,T];H_0^1(\Omega)), \mbox{ with } \dot{u}\in L^2((0,T];L^2(\Omega))$, $\ddot{u}\in L^2((0,T];H^{-1}(\Omega)),$
is a weak solution of the hyperbolic initial boundary value problem (\ref{setup})--(\ref{bc}) if
\begin{enumerate}[(i)]
\item $\left\langle\ddot{u}(t),\varphi\right\rangle + B[u,\varphi; t]=\left(f(t),\varphi \right)_{L^2} \mbox{ for all } \varphi\in H_0^1(\Omega)$ and $0<  t \leq T$ and
\item $u(x,0)=u_0(x)$, $\dot{u}(x,0)=u_1(x).$
\end{enumerate}
\end{definition}
\begin{remark}\label{conitnuity}
It is well known that problem (\ref{setup})--(\ref{bc}) is well-posed and admits a unique weak solution (see Theorem 29.1 in \cite{Wl}). By Theorem $2$ in Section $5.9.2$ \cite{Ev}, we know that 
$$u\in C([0,T];L^2(\Omega)) \mbox{ and } \dot{u}\in C([0,T];H^{-1}(\Omega)).$$
This shows that the initial conditions are meaningfully attained in $L^2(\Omega)$ and $H^{-1}(\Omega)$ respectively. Furthmore, it is even true that (see Lions and Magenes \cite{LM}, Chapter III, Theorems 8.1 and 8.2)
$$u\in C([0,T];H_0^1(\Omega)) \mbox{ and } \dot{u}\in C([0,T];L^2(\Omega)).$$
Consequently, the initial conditions stated as (ii) in the above definition make sense.
\end{remark}
\subsection{Discontinuous-in-time  discretization}\label{semi-discrete}
We partition the interval $I=(0,T]$ into $N$ sub-intervals $I_n=(t_{n-1}, t_n]$ having length $k_n=t_n-t_{n-1}$ for $n=1,2,\ldots, N$ with $t_0=0$ and $t_N=T$, as shown below.
\begin{figure}[htp]
\includegraphics[scale=0.6]{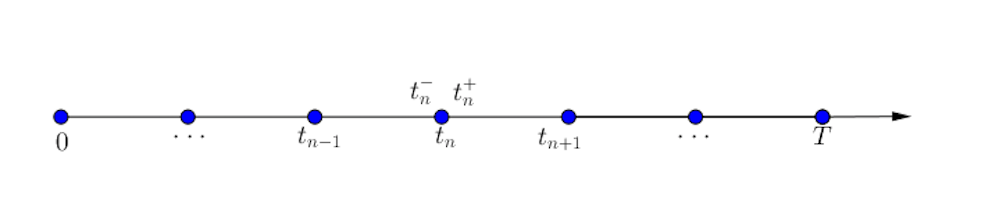}
\caption{Partition of the time domain}
\end{figure}
To deal with the discontinuity at each $t_n$ in the numerical approximation to $u$, we introduce the jump operator 
$$[v]_n := v(t_n^{+})-v(t_n^{-}) \mbox { for } n= 0, 1,\ldots,N-1,$$
where
$$v(t_n^{\pm})=\lim_{\varepsilon\to 0^{\pm}} v(t_n+\varepsilon) \mbox { for } n=0, 1,\ldots,N-1,$$
for a function $v$. By convention, we assume  that $v(0^{-})=u_0$ and $\dot{v}(0^{-})=u_1.$ Thus $$[v]_0=v(0^{+})-u_0, \mbox { and } [\dot{v}]_0=\dot{v}(0^{+})-u_1.$$ Moreover, we shall write $v_n^{+}$ for $v(t_n^{+})$ and $v_n^{-}$ for $v(t_n^{-})$.

Now we focus on the generic time interval $I_n$ and assume that the solution on $I_{n-1}$ is known. Testing the equation (\ref{setup}) against $\dot{v}$ where $v\in H^1(I_n; H_0^1(\Omega))$, we have 
\begin{equation}\label{case1}
\int_{t_{n-1}}^{t_n}\: \left\langle\ddot{u},\dot{v}\right\rangle\, \dd t+\int_{t_{n-1}}^{t_n}\: \left(\dot{u},\dot{v}\right)_{L^2}\, \dd t+ \int_{t_{n-1}}^{t_n}\: \left(\nabla u, \nabla \dot{v} \right)_{L^2}\, \dd t= \int_{t_{n-1}}^{t_n}\: (f,\dot{v})_{L^2}\, \dd t.
\end{equation}
Since $[u]_n=[\dot{u}]_n=0$ for $n=0,1,\ldots, N-1$, we can rewrite (\ref{case1}) by adding strongly consistent terms:
\begin{flalign}\label{case1extra}
&\int_{t_{n-1}}^{t_n}\:\left\langle\ddot{u},\dot{v}\right\rangle\, \dd t +\int_{t_{n-1}}^{t_n}\: \left(\dot{u},\dot{v}\right)_{L^2}\, \dd t+\int_{t_{n-1}}^{t_n} \:\left(\nabla u,\nabla \dot{ v} \right)_{L^2}\, \dd t+([\dot{u}]_{n-1}, \dot{v}_{n-1}^{+})_{L^2}+(\nabla [u]_{n-1},\nabla v_{n-1}^{+})_{L^2} \nonumber\\
&= \int_{t_{n-1}}^{t_n} (f,\dot{v})_{L^2} \:\dd t \mbox { for } n=1,\ldots,N. &&
\end{flalign}
Note that we define by convention $u_0^{-}:=u_0$ and $\dot{u}_{0}^{-}:=u_1.$
Summing over all time intervals in (\ref{case1extra}), we are able to define the bilinear forms $\mathcal{A}\colon \mathcal{H}\times \mathcal{H} \to\mathbb{R}$ with 
$$\mathcal{H}:= H^2((0,T];H^{-1}(\Omega))\cap H^1((0,T];H_0^1(\Omega))\cap L^2((0,T];H_0^1(\Omega))$$
by
\begin{equation}\label{case1blinear}
\begin{split}
\mathcal{A}(u,v):=&\sum_{n=1}^N\int_{t_{n-1}}^{t_n}\: \left\langle\ddot{u},\dot{v}\right\rangle \,\dd t +\sum_{n=1}^N\int_{t_{n-1}}^{t_n}\: \left(\dot{u},\dot{v}\right)_{L^2}\, \dd t+\sum_{n=1}^N\int_{t_{n-1}}^{t_n}\: \left(\nabla u, \nabla \dot{v}\right)_{L^2} \dd t\\
{}&+\sum_{n=1}^{N-1}([\dot{u}]_{n}, \dot{v}_{n}^{+})_{L^2}+\sum_{n=1}^{N-1}(\nabla [u]_{n}, \nabla v_{n}^{+})_{L^2} +(\dot{u}_0^{+}, \dot{v}_0^{+})_{L^2}+(\nabla u_0^{+},\nabla v_{0}^{+})
\end{split}
\end{equation}
for all $v\in \mathcal{H}.$
The linear functional $\mathcal{F}\colon \mathcal{H}\to\mathbb{R}$ is defined by
\begin{equation}\label{LF}
\mathcal{F}(v):=\sum_{n=1}^N\int_{t_{n-1}}^{t_n} (f,\dot{v})_{L^2}\:\dd t +(u_1,\dot{v}_0^{+})_{L^2}+(\nabla u_0,\nabla v_0^{+})_{L^2}.
\end{equation}
Next, we introduce the local semi-discrete space
$$\mathcal{V}^{q_n}:=\{v\in L^2((0,T]; H_0^1(\Omega))\colon v\mid_{I_n}\in\mathbb{P}^{q_n}(I_n; H_0^1(\Omega))\},$$
where $\mathbb{P}^{q_n}(I_n; H_0^1(\Omega))$ is the space of polynomials in $t$ of degree less than or equal to $q_n\geq 2$ on $I_n$ with coefficients in $H_0^1(\Omega)$. Then, introducing $\mathbf{q}:=(q_1,q_2,\ldots,q_N)\in\mathbb{N}^{N}$, the $N$-dimensional polynomial degree vector, we can define the discontinuous Galerkin finite element space as 
$$\mathcal{V}^{\mathbf{q}}:=\{v\in L^2((0,T]; H_0^1(\Omega)) \colon v\mid_{I_n}\in \mathcal{V}^{q_n} \mbox{ for all } n=1,2,\ldots, N \}.$$
The discontinuous-in-time formulation of problem (\ref{setup})--(\ref{bc}) reads as follows: find $u_{\mathrm{DG}}\in\mathcal{V}^{\mathbf{q}}$ such that
\begin{equation}\label{DG_formulation}
\mathcal{A}(u_{\mathrm{DG}},v)=\mathcal{F}(v),  \mbox{ for all } v\in\mathcal{V}^{\mathbf{q}}.
\end{equation}
\subsection{Stability analysis}
We begin by defining an energy norm, which will be used in the stability and convergence analysis.
\begin{proposition}\label{norm}
The function $|||\cdot||| \colon \mathcal{V}^{\mathbf{q}}\to\mathbb{R}^{+}$ defined by 
\begin{equation}\label{case1norm}
\begin{split}
|||v|||^2:=& \:\frac{1}{2}\|\dot{v}_0^{+}\|_{L^2}^2+\frac{1}{2}\sum_{n=1}^{N-1}\|[\dot{v}]_n\|_{L^2}^2+\frac{1}{2}\|\dot{v}_{N}^{-}\|_{L^2}^2+\sum_{n=1}^N \int_{t_{n-1}}^{t_n} \|\dot{v}\|_{L^2}^2\:\dd t\\
{}&+\frac{1}{2}\|\nabla v_0^{+}\|_{L^2}^2+\frac{1}{2}\sum_{n=1}^{N-1}\|\nabla [v]_n\|_{L^2}^2+\frac{1}{2}\|\nabla v_{N}^{-}\|_{L^2}^2
\end{split}
\end{equation}
is a norm on $\mathcal{V}^{\mathbf{q}}.$
\begin{proof}
The homogeneity and triangle inequality follow from the properties of the $L^2$ and $H_0^1$ norms. It is sufficient to show that 
$$|||v|||=0 \Leftrightarrow v\equiv 0.$$
If $v\equiv0$, it is trivially true that $|||v|||=0$. Now we need to show that $|||v|||=0\Rightarrow v\equiv 0.$
If $|||v|||=0$, then $\|v_0^{+}\|_{H_0^1}:=\|\nabla v_0^{+}\|_{L^2}^2=0$ and $\int_{t_{0}}^{t_1}\|\dot{v}\|_{L^2}^2 \, \dd t=0$. We have 
$$\begin{cases}
\dot{v}(t) =0  & \mbox{ for } t\in I_1,\\
v_0^{+}=0.\\
\end{cases}$$
This implies that $v(t)\equiv 0$ on $I_1$. We now proceed by induction. Assume that $v(t)\equiv 0$ on $I_{n-1}$ and consider the time interval $I_n$. From $\|[v]_{n-1}\|_{H_0^1}:=\|\nabla [v]_{n-1}\|_{L^2}=0$, we know that $v_{n-1}^{+}=v_{n-1}^{-}=0$. Then 
$$\begin{cases}
\dot{v}(t) =0  & \mbox{ for } t\in I_n,\\
v_{n-1}^{+}=0.\\
\end{cases}$$
Thus, we conclude that $v\equiv 0$ on each interval $I_n$ for $n=1,\ldots, N$. That is, $v\equiv 0$ on $(0,T]$.
\end{proof}
\end{proposition}
\begin{remark}
If we take $u=v$ in (\ref{case1blinear}), we have 
\begin{equation}\label{Variational Form}
\mathcal{A}(v,v)=|||v|||^2\mbox{ for all } v\in \mathcal{V}^{\mathbf{q}}. 
\end{equation}
This shows that the bilinear form $\mathcal{A}(\cdot,\cdot)$ is coercive with respect to the $|||\cdot|||$ norm with coercivity constant $\alpha=1$.
\end{remark}
\begin{theorem}
Let $f\in L^2((0,T];L^2(\Omega))$, $u_0\in H_0^1(\Omega)$ and $u_1\in L^2(\Omega)$. Then the solution of (\ref{DG_formulation}), $u_{\mathrm{DG}}\in\mathcal{V}^r$ satisfies 
$$|||u_{\mathrm{DG}}|||\leq\left(\|f\|_{L^2((0,T];L^2(\Omega))}^2+2\| \nabla u_0\|_{L^2}^2+2\|u_1\|_{L^2}^2\right)^{\frac{1}{2}}.$$
\begin{proof}
Using the definition of $|||\cdot|||$ in (\ref{case1norm}) and Young's inequality, we have
\begin{flalign*}
|||u_{\mathrm{DG}}|||^2=&\:\mathcal{A}(u_{\mathrm{DG}},u_{\mathrm{DG}})\\
{}=& \:F(u_{\mathrm{DG}})&&\\
{}=&\: \sum_{n=1}^N\int_{t_{n-1}}^{t_n} (f,\dot{u}_{\mathrm{DG}})_{L^2}\: \dd t+(u_1,\dot{u}_{DG}(\cdot, 0^{+}))_{L^2}+(\nabla u_0, \nabla u_\mathrm{DG}(\cdot,0^{+}))_{L^2}&&\\
{}\leq&\: \frac{1}{2}\sum_{n=1}^N\int_{t_{n-1}}^{t_n} \|f(\cdot,t)\|_{L^2}^2 \:\dd t+\frac{1}{2}\sum_{n=1}^N\int_{t_{n-1}}^{t_n}\|\dot{u}_{\mathrm{DG}}(\cdot,t)\|_{L^2}^2 \:\dd t+\|u_1\|_{L^2}^2+\frac{1}{4}\|\dot{u
}_{\mathrm{DG}}(\cdot,0^{+})\|_{L^2}^2&&\\
&+\|\nabla u_0\|_{L^2}^2+\frac{1}{4}\|\nabla u_{\mathrm{DG}}(\cdot, 0^{+})\|_{L^2}^2&&\\
{}\leq &\: \frac{1}{2}\sum_{n=1}^N\int_{t_{n-1}}^{t_n} \|f(\cdot,t)\|_{L^2}^2 \:\dd t+\frac{1}{2}|||u_{\mathrm{DG}}|||^2+\|u_1\|_{L^2}^2+\|\nabla u_0\|_{L^2}^2.&&
\end{flalign*}
Then the required inequality follows.
\end{proof}
\end{theorem}
\subsection{Convergence analysis}
In this section, we prove \emph{a priori} error estimates with respect to the energy norm $|||\cdot|||$ defined in (\ref{case1norm}).

Before we proceed to the proof of the main convergence result, let us first define the $L^2$-projection operator based on Definition 3.1 in \cite{SS} and Definition 3.2 in \cite{AMSQ}.
\begin{definition}\label{l2projdef}
Let $I=(-1,1)$. For a function $w\in C((-1,1];H_0^1(\Omega))\cap L^2(I; H_0^1(\Omega))$, we define the (boundary value preserving) $L^2$-projection $\mathcal{P}^q w\in\mathbb{P}^q(I;H_0^1(\Omega))$ with $q\geq 0$ by the following conditions 
\begin{equation}\label{zeroatend}
\mathcal{P}^q w(x,1)=w(x,1)\in H_0^1(\Omega),
\end{equation}
\begin{equation}\label{orthogonality}
\int_{I} \left((w-\mathcal{P}^q w), \varphi\right)_{L^2} \:\dd t=0 \mbox{ for all } \varphi\in\mathbb{P}^{q-1}(I;H_0^1(\Omega)).
\end{equation}
When $q=0$, only the first condition is necessary.
\end{definition}

\begin{definition}\label{proj}
Let $I=(-1,1)$. For a function $u\in H^1(I; H_0^1(\Omega))$ such that $\dot{u}\in C((-1,1]; H_0^1(\Omega))$, we define $\Pi^q u\in\mathbb{P}^q(I; H_0^1(\Omega))$ with $q\geq 1$ by 
$$\Pi^q u(x,t)=u(x,-1)+\int_{-1}^t \mathcal{P}^{q-1} \dot{ u}(x,s)\,\dd s\mbox{ for all } t\in (-1,1],$$
where $\mathcal{P}^{q-1}\dot{u}$ is given in Definition \ref{l2projdef}.
\end{definition}
Now we need to show that $\Pi^q u(x,t)$ in Definition \ref{proj} is well-defined. Consider the sequence of Legendre polynomials $\{L_i\}_{i\geq 0}$, $L_i\in\mathbb{P}^i(I)$ on $I=(-1,1)$ defined by the following recurrence relation
$$(i+1)L_{i+1}(t)=(2i+1)tL_i(t)-iL_{i-1}(t)$$
with $L_0(t)=1$ and $L_1(t)=t$.

The relevant properties of Legendre polynomials for our purpose are the following:
\begin{enumerate}
\item $L_i(1)=1 \mbox{ for } i\in \mathbb{N}\cup \{0\};$
\item $\int_{-1}^t L_i(s)\,\dd s=\frac{1}{2i+1} (L_{i+1}(t)-L_{i-1}(t))\mbox{ for } i\in \mathbb{N};$
\item $\int_{-1}^1 L_i(t)^2\,\dd t=\frac{2}{2i+1}\mbox{ for } i\in\mathbb{N}\cup \{0\}.$
\end{enumerate}
\begin{lemma}\label{projlemma}
$\Pi^q u(x,t)$ defined in Definition \ref{proj} is well-defined.
\begin{proof}
\emph{Existence}: Note that $\dot{u}\in L^2(I; H_0^1(\Omega))$, so we can write $\dot{ u}$ as a Legendre series in the following form
$$\dot{ u}(x,t)=\sum_{i=0}^{\infty} b_i(x) L_i(t)$$
with $b_i\in H_0^1(\Omega)$ for each $i\in\mathbb{N}\cup \{0\}.$ From Lemma 3.5 in \cite{SS}, we find that 
$$\mathcal{P}^{q-1} \partial_t u(x,t)=\sum_{i=0}^{q-2} b_i(x)L_i(t) +\left(\sum_{i=q-1}^{\infty} b_i(x)\right) L_{q-1}(t),$$
and consequently,
\begin{flalign*}
\Pi^q u(x,t)=&\:u(x,-1)+\sum_{i=0}^{q-2} b_i(x) \int_{-1}^t L_i(s) \,\dd s+ \left(\sum_{i=q-1}^{\infty} b_i(x)\right)\int_{-1}^t L_{q-1}(s)\,\dd s &&\\
{}=&\: u(x,-1)+b_0(x)(t+1)+\sum_{i=1}^{q-2} b_i(x) \frac{1}{2i+1}(L_{i+1}(t)-L_{i-1}(t)) &&\\
{}& +\left(\sum_{i=q-1}^{\infty} b_i(x) \right) \left(\frac{1}{2q-1}(L_q(t)-L_{q-2}(t))\right) &&\\
{}=& \:\left[ u(x,-1)+b_0(x)-\frac{b_1(x)}{3}\right]L_0(t)+\left[b_0(x)-\frac{b_2(x)}{5}\right]L_1(t)+\sum_{i=2}^{q-3}\left[\frac{b_{i-1}(x)}{2i-1}-\frac{b_{i+1}(x)}{2i+3}\right] L_i(t)&&\\
&+\left[\frac{b_{q-3}(x)}{2q-5}-\sum_{i=q-1}^{\infty}\frac{b_i(x)}{2q-1} \right]L_{q-2}(t)+\frac{b_{q-2}(x)}{2q-3}L_{q-1}(t)+\sum_{i=q-1}^{\infty}\frac{b_i(x)}{2q-1}L_q(t), &&
\end{flalign*}
or equivalently
\begin{equation}\label{expansion}
\Pi^q u(x,t)=\sum_{i=0}^q u_i^{\ast}(x)L_i(t),
\end{equation}
\begin{align*}
u_0^{\ast}(x)=&\:u(x,-1)+b_0(x)-\frac{b_1(x)}{3},\\
u_1^{\ast}(x)=&\:b_0(x)-\frac{b_2(x)}{5},\\
u_i^{\ast}(x)=&\:\frac{b_{i-1}(x)}{2i-1}-\frac{b_{i+1}(x)}{2i+3} \mbox{ for }i=2,\ldots, q-3,\\
u_{q-2}^{\ast}=&\:\frac{b_{q-3}(x)}{2q-5}-\sum_{i=q-1}^{\infty}\frac{b_i(x)}{2q-1},\\
u_{q-1}^{\ast}=&\:\frac{b_{q-2}(x)}{2q-3},\\
u_{q}^{\ast}=&\sum_{i=q-1}^{\infty}\frac{b_i(x)}{2q-1}.
\end{align*}
\emph{Uniqueness}:  Assume for contradiction that both $u_1, u_2\in \mathbb{P}^q(I,H_0^1(\Omega))$ satisfy the conditions in the definition. In particular, we have $u_1(x,1)=u_2(x,1)$. Now considering the difference $u_1-u_2$, we can write it as a Legendre series $u_1-u_2=\sum_{i=0}^q a_i(x) L_i(t)$ with $a_i=\int_{I} (u_1-u_2) L_i \,\dd t\in H_0^1(\Omega)$. It follows from the orthogonality condition that $$\int_{I} \left (u_1-u_2, a L_k\right)_{L^2} \,\dd t=0 \mbox{ for all } a\in H_0^1(\Omega),\quad 0\leq k \leq q-1.$$
Using the orthogonality properties of the Legendre polynomials we get $\left( a_k, a\right)_{L^2}=0$ for all $a\in H_0^1(\Omega)$. Since $H_0^1(\Omega)$ is dense in $L^2(\Omega)$, we have that $a_k=0$ in $L^2(\Omega)$ and thus $a_k=0$ in $H_0^1(\Omega)$ for $0\leq k\leq q-1.$ This implies that $u_1-u_2=a_q(x) L_q(t)$. Since $u_1(x,1)=u_2(x,1)$, we have that  $a_q\equiv 0,$ which proves the uniqueness of a projection polynomial satisfying the conditions in Definition \ref{proj}.
\end{proof}
\end{lemma}
\begin{lemma}\label{projproperties}
For any $u\in H^1(I;H_0^1(\Omega))$ such that $\dot{u}\in C((-1,1]; H_0^1(\Omega))$, we have for $q\geq 2$,
\begin{enumerate}
\item $(\Pi^q u- u)(x,-1)=0$;
\item $(\Pi^q u- u)(x,1)=0$;
\item $\partial_t(\Pi^q u-u)(x,1)=0$;
\item $\int_{I} \left(\partial_t\left( u-\Pi^q u\right), \varphi\right)_{L^2}\,\dd t=0 \mbox{ for all } \varphi\in\mathbb{P}^{q-2}(I; H_0^1(\Omega))$.

\begin{proof}
(a) follows directly from Definition \ref{proj}. For $(b)$, first note that we can expand $\dot{u}$ as $$\dot{u}=\sum_{i=0}^{\infty} b_i(x) L_i(t)$$ with coefficients $b_i\in H_0^1(\Omega)$. Then $u$ can be written as 
\begin{equation}\label{ueq}
u(x,t)=u(x,-1)+\sum_{i=0}^{\infty} b_i(x)\int_{-1}^t L_i (s) \:\dd s.
\end{equation}
By orthogonality of the Legendre polynomials, we have 
\begin{equation}\label{urelation}
u(x,1)=u(x,-1)+\sum_{i=0}^{\infty} b_i(x) \int_{-1}^1 L_i(s)\: \dd s= u(x,-1)+2b_0.
\end{equation}
Then 
\begin{align*}
\Pi^q u(x,1)&=u(x,-1)+\int_{-1}^1 \mathcal{P}^{q-1} \dot{u}(x,s)\,\dd s\\
{}&= u(x,-1)+\sum_{i=0}^{q-2}b_i(x)\int_{-1}^1 L_i(s)\,\dd s+\left( \sum_{i=q-1}^{\infty} b_i(x) \right) \int_{-1}^1 L_{q-1} (s) \, \dd s\\
{}&=u(x,-1)+2b_0 \\
{}&= u(x,1) \mbox{ (by (\ref{urelation}))}.
\end{align*}
(c) follows from taking the derivative with respect to $t$ in the equation
$$\Pi^q u(x,t)=u(x,-1)+\sum_{i=0}^{q-2} b_i(x)\int_{-1}^t L_i(s)\,\dd s +\left(\sum_{i=q-1}^{\infty} b_i(x)\right) \int_{-1}^t L_{q-1}(s)\,\dd s.$$ Note that $$\dot{ u}(x,t)=\sum_{i=0}^{\infty} b_i(x) L_i(t)$$
and $$\partial_t \Pi^q u(x,t)=\sum_{i=0}^{q-2} b_i(x) L_i(t)+\left(\sum_{i=q-1}^{\infty} b_i(x)\right) L_{q-1}(t).$$
Evaluating at $t=1$, we have $$\partial_t \Pi^q u(x,1)=\sum_{i=0}^{\infty} b_i(x)=\dot{ u}(x,1).$$ For (d), we can write any $q\in \mathbb{P}^{q-2}(I;H_0^1(\Omega))$ as $q(x,t)=\sum_{k=0}^{q-2} a_k(x) L_k(t)$. Then
\begin{align*}
&\int_{I} \left(\partial_t (u-\Pi^q u ), \varphi\right)_{L^2} \,\dd t\\
&=\int_{I} \sum_{k=0}^{q-2} \left(\partial_t (u-\Pi^q u ), a_kL_k\right)_{L^2} \,\dd t\\
&=\sum_{k=0}^{q-2}  \int_{I} \left(\partial_t u-\mathcal{P}^{q-1}\partial_t u) , a_k L_k\right)_{L^2}\,\dd t\\
&=\sum_{k=0}^{q-2}  \int_{I} \left(\left(\sum_{i=q}^{\infty} b_i(x) L_i(t)-\left(\sum_{i=q}^{\infty} b_i (x)\right)L_{q-1}(t) \right), a_k L_k(t)\right)_{L^2} \,\dd t\\
&= 0 
\end{align*}
by the orthogonality of Legendre polynomials.
\end{proof}
\end{enumerate}
\end{lemma}
On an arbitrary time interval $I=(a,b)$, we define $\Pi_{I}^q$ via the linear map
$$T\colon \Omega\times (-1,1)\to \Omega\times (a,b)$$
$$(x,\xi)\mapsto (x,\frac{1}{2} (a+b+\xi (b-a)))$$
as \begin{equation}\label{projope}
\Pi_{I}^q u=[\Pi^q(u\circ T)]\circ T^{-1},
\end{equation} from which it follows that 
\begin{enumerate}
\item $(\Pi_{I}^q u-u)(x,a)=0;$
\item $(\Pi_{I}^q u-u)(x,b)=0;$
\item $\partial_t (\Pi_{I}^q u-u)(x,b)=0;$
\item $\int_{I} \left(\partial_t (u-\Pi_{I}^q u), \varphi\right)_{L^2} \,\dd t=0 \mbox{ for all } \varphi\in\mathbb{P}^{q-2}(I; H_0^1(\Omega))$.
\end{enumerate}
Similarly as before, on the generic interval $I=(a,b)$, $\Pi_{I}^q u$ can be written as
$$\Pi_{I}^q u(x,t)=\sum_{i=0}^q \tilde{u}_{i}^{\ast} \tilde{L_i}(t)=u(x,a)+\sum_{i=0}^{q-2} \tilde{b}_i(x)\int_{a}^t\tilde{L}_i(s)\,\dd s+\left( \sum_{i=q-1}^{\infty} \tilde{b}_i\right) \int_{a}^t \tilde{L}_{q-1}(s)\, \dd s,$$
where $\tilde{L}_i$ represents the $i$th mapped Legendre polynomial so that $u=\sum_{i=0}^{\infty} \tilde{u}_{i}\tilde{L}_i$ and the coefficients are defined as before but with $b_i$ replaced by $\tilde{b}_i$ for each $i\in\mathbb{N}\cup \{0\}$.

Now we generalize a standard approximation result stated in \cite{BS} by Babu\v{s}ka and Suri to functions in Bochner spaces.

\begin{proposition}\label{sequenceapproximate}
Let $I=(a,b)$ with $k=b-a>0$. For every $u\in H^{s}(I; H_0^1(\Omega))$, there exists a sequence $\{\mathcal{Q}^q u\}_{q\geq 0}$ with $\mathcal{Q}^q u\in\mathbb{P}^q(I;H_0^1(\Omega))$ such that for any $0\leq m\leq s$,
\begin{equation}\label{Hq}
\| u-\mathcal{Q}^q u\|_{H^m(I;H_0^1(\Omega))}\leq\: C\frac{k^{\mu-m}}{q^{s-m}}\|u\|_{H^s(I;H_0^1(\Omega))} \mbox{ for } s\geq 0,
\end{equation}
where $\mu=\min (q+1, s)$ and $C$ is a constant independent of $u$, $q$ and $k.$
\end{proposition}
Next, we define the projector $\mathcal{P}_I^{q} u$ on an arbitrary interval $I=(a,b)$ via the same linear map $T\colon \Omega\times (-1,1)\rightarrow \Omega \times (a,b)$ as before. That is, 
$$\mathcal{P}_{I}^q u=[\mathcal{P}^q(u\circ T)]\circ T^{-1}.$$
Then the following approximation result holds.
\begin{proposition}\label{estprop}
Let $I=(a,b)$ with $k=b-a>0$. For every $u\in H^s(I; H_0^1(\Omega))$, $s\geq 2$, we have 
\begin{equation}
\|\partial_t u-\mathcal{P}_{I}^{q-1}(\partial_t u)\|_{L^2(I; H_0^1(\Omega))}\leq \: C\frac{k^{\mu-1}}{q^{s-1}}\|u\|_{H^s(I; H_0^1(\Omega))},
\end{equation}
where $\mu=\min(q+1,s)$ and $C$ is a universal constant.
\begin{proof}
We give a sketch of the proof based on several results from \cite{SS} and \cite{Sc}. For $u\in H^s(I; H_0^1(\Omega))$, define $\hat{u}(x,\xi)$ for $\xi \in \hat{I}:=(-1,1)$  and $t=\frac{1}{2}(a+b+\xi k)$ by 
$$u(x,t)=u(x, \frac{1}{2}(a+b+\xi k)):=\hat{u}(x,\xi).$$
Then we have $\partial_{\xi} \hat{u}(x,\xi)=\frac{k}{2}\partial_t u(x,t).$ By Lemma \ref{projlemma}, we have 
$$\partial_{\xi} \hat{u}(x,\xi)=\sum_{i=0}^{\infty} b_i(x)L_i(\xi)$$ and 
$$\mathcal{P}^{q-1} \partial_{\xi} \hat{u}(x,\xi)=\sum_{i=0}^{q-2} b_i(x) L_i(\xi)+\left(\sum_{i=q-1}^{\infty} b_i(x) \right)L_{q-1}(\xi)$$
where $\{L_i\}_{i\geq 0}$ are Legendre polynomials defined on $\hat{I}:=(-1,1).$
This implies that 
$$\partial_{\xi} \hat{u}(x,\xi)-\mathcal{P}^{q-1}\partial_{\xi} \hat{u}(x,\xi)=\sum_{i=q}^{\infty} b_i(x)L_i(\xi)-\left(\sum_{i=q}^{\infty} b_i(x) \right) L_{q-1}(\xi).$$
Using the triangle inequality, we obtain
$$\|\partial_{\xi} \hat{u}-\mathcal{P}^{q-1} \partial_{\xi} \hat{u}\|_{L^2(\hat{I}; H_0^1(\Omega))}\leq\|\sum_{i=q}^{\infty} b_i(x)L_i(\xi) \|_{L^2(\hat{I}; H_0^1(\Omega))}+\|\left( \sum_{i=q}^{\infty} b_i(x)\right)L_{q-1}(\xi)\|_{L^2(\hat{I}; H_0^1(\Omega))}.$$
Note that by Lemma 3.9 from \cite{Sc}, we know that  
$$\|\partial_{\xi} \hat{u}-\mathbf{P}^{q-1} \partial_{\xi} \hat{u}\|_{L^2(\hat{I}; H_0^1(\Omega))}= \|\sum_{i=q}^{\infty} b_i(x)L_i(\xi) \|_{L^2(\hat{I}; H_0^1(\Omega))},$$
where $\mathbf{P}^{q-1}$ is the standard $L^2$-projection operator such that $$\|\partial_{\xi} \hat{u}-\mathbf{P}^{q-1} \partial_{\xi} \hat{u}\|_{L^2(\hat{I}; H_0^1(\Omega))}=\inf_{\varphi\in\mathbb{P}^{q-1}(\hat{I}; H_0^1(\Omega))}\|\partial_{\xi} \hat{u} -\varphi\|_{L^2(\hat{I}; H_0^1(\Omega))}.$$
Using the orthogonality of Legendre polynomials on $\hat{I}$, we have 
$$\| \sum_{i=q}^{\infty} b_i(x) L_i(\xi) \|_{L^2(\hat{I}; H_0^1(\Omega))}^2=\sum_{i=q}^{\infty} \|b_i\|_{H_0^1(\Omega)}^2 \frac{2}{2i+2}.$$ 
By Lemma 3.10 from \cite{Sc}, we know that for $1\leq s_0\leq s$ we have 
\begin{equation}\label{weightednorm}
\int_{-1}^1 | \hat{u}^{(s_0)}(x,\xi)|^2(1-\xi^2)^{s_0-1} \,\dd \xi=\sum_{i=s_0-1}^{\infty}|b_i(x)|^2 \frac{2}{2i+1}\frac{(i+(s_0-1))!}{(i-(s_0-1))!}
\end{equation}
where $\hat{u}^{(s_0)}(x,\xi)$ represents the $s_0$-th partial derivative of $\hat{u}(x,\xi)$ with respect to $\xi$.
Thus 
\begin{align*}
\|\sum_{i=q}^{\infty}  b_i(x) L_i(\xi)\|_{L^2(\hat{I};H_0^1(\Omega))}^2\leq & \:\frac{(q-(s_0-1))!}{(q+(s_0-1))!}\sum_{i=q}^{\infty} \|b_i\|_{H_0^1(\Omega)}^2\frac{2}{2i+2}\frac{(i+(s_0-1))!}{(i-(s_0-1))!}\\
{}\leq & \:\frac{(q-(s_0-1))!}{(q+(s_0-1))!} |\hat{u}|_{H^{s_0}(\hat{I}; H_0^1(\Omega))}^2 \mbox{ (by (\ref{weightednorm}))}
\end{align*}
for any $1\leq s_0\leq \min(q+1,s).$
Thus, we have 
\begin{align*}
\inf_{\varphi\in\mathbb{P}^{q-1}(\hat{I}; H_0^1(\Omega))}\|\partial_{\xi} \hat{u} -\varphi\|_{L^2(\hat{I}; H_0^1(\Omega))}^2&=\|\sum_{i=q}^{\infty}b_i(x) L_i(\xi)\|_{L^2(I;H_0^1(\Omega))}^2\\
{}&\leq \frac{(q-(\mu-1))!}{(q+(\mu-1))!} |\hat{u}|_{H^{\mu}(\hat{I}; H_0^1(\Omega))}^2
\end{align*}
where $\mu=\min(q+1,s).$  Note also that
\begin{align*}
\|\left(\sum_{i=q}^{\infty} b_i(x) \right) L_{q-1}(\xi) \|_{L^2(\hat{I}; H_0^1(\Omega))}^2=&\:\frac{2}{2q-1}\|\sum_{i=q}^{\infty} b_i(x) \|_{H_0^1(\Omega)}^2\\
{}\leq &\: \frac{4}{q}  \|\sum_{i=q}^{\infty} b_i(x) \|_{H_0^1(\Omega)}^2
\end{align*}
for $q\geq 2.$ From Lemma 3.6 in \cite{SS}, we know that 
 $$\|\sum_{i=q}^{\infty} b_i(x) \|_{H_0^1(\Omega)}^2\leq \: \frac{C_2}{q} \|\partial_{\xi \xi} \hat{u}\|_{L^2(\hat{I}; H_0^1(\Omega))}^2,$$
where $\partial_{\xi} \hat{u}(x,\xi)=\sum_{i=0}^{\infty} b_i(x) L_i(\xi)$ and $C_2$ is a positive constant.
Combining with the previous estimate, we have 
\begin{align}\label{ieq}
\|\partial_{\xi} \hat{u}-\mathcal{P}^{q-1} \partial_{\xi} \hat{u}\|_{L^2(\hat{I}; H_0^1(\Omega))}^2\leq &\: C\left\{\|\partial_{\xi} \hat{u} -\mathbf{P}^{q-1}\partial_{\xi} \hat{u}\|_{L^2(\hat{I}; H_0^1(\Omega))}^2+\frac{1}{q^2}\|\partial_{\xi\xi} \hat{u}\|_{L^2(\hat{I}; H_0^1(\Omega))}^2\right\}
\end{align}
for some constant $C>0.$
If we replace $\partial_{\xi} \hat{u}$ by $\partial_{\xi} \hat{u}-\varphi$ in (\ref{ieq}) for an arbitrary $\varphi\in\mathbb{P}^{q-1}(\hat{I};H_0^1(\Omega))$, then we have 
\begin{equation}\label{ieq2}
\|\partial_{\xi} \hat{u}-\mathcal{P}^{q-1} \partial_{\xi} \hat{u}\|_{L^2(\hat{I}; H_0^1(\Omega))}^2\leq \: C\left\{\inf_{\varphi\in\mathbb{P}^{q-1}(\hat{I}; H_0^1(\Omega))}\|\partial_{\xi} \hat{u} -\varphi\|_{L^2(\hat{I}; H_0^1(\Omega))}^2+\frac{1}{q^2}\|\partial_{\xi\xi} \hat{u} -\partial_{\xi} \varphi\|_{L^2(\hat{I}; H_0^1(\Omega))}^2 \right\}
\end{equation}
using the fact that 
$$\| \partial_{\xi} \hat{u}-\mathbf{P}^{q-1}\partial_{\xi}  \hat{u}\|_{L^2(\hat{I};H_0^1(\Omega))}\leq \|\partial_{\xi} \hat{u}-\varphi\|_{L^2(I; H_0^1(\Omega))} \mbox{ for any } \varphi\in\mathbb{P}^{q-1}(I; H_0^1(\Omega)).$$
Similarly, we can show that 
\begin{align*}
\|\partial_{\xi\xi}\hat{u} -\partial_{\xi} \varphi\|_{L^2(\hat{I}; H_0^1(\Omega))}^2\leq& \:\frac{(q-(s_0-1))!}{(q+(s_0-1))!}|\partial_{\xi} \hat{u}|_{H^{s_0}(\hat{I}; H_0^1(\Omega))}^2\\
{}=&\:\frac{(q-(s_0-1))!}{(q+(s_0-1))!}|\hat{u}|_{H^{s_0+1}(\hat{I}; H_0^1(\Omega))}^2
\end{align*}
for $1\leq s_0\leq \min(q+1,s-1)$. Thus, 
$$\|\partial_{\xi\xi}\hat{u} -\partial_{\xi} \varphi\|_{L^2(\hat{I}; H_0^1(\Omega))}^2\leq\:\frac{(q-(\mu-2))!}{(q+(\mu-2))!}|\hat{u}|_{H^{\mu}(\hat{I}; H_0^1(\Omega))}^2$$
for $\mu=\min(q+2,s)$.
Therefore,
\begin{align*}
\|\partial_{\xi} \hat{u}-\mathcal{P}^{q-1}\partial_{\xi} \hat{u}\|_{L^2(\hat{I};H_0^1(\Omega))}^2\leq & \:C\left\{\frac{(q-(\mu-1))!}{(q+(\mu-1))!} |\hat{u}|_{H^{\mu}(\hat{I}; H_0^1(\Omega))}^2+\frac{1}{q^2}\frac{(q-(\mu-2))!}{(q+(\mu-2))!} |\hat{u}|_{H^{\mu}(\hat{I}; H_0^1(\Omega))}^2\right\}\\
{}\leq & \:\tilde{C} \frac{1}{q^2}\frac{(q-(\mu-2))!}{(q+(\mu-2))!}|\hat{u}|_{H^{\mu}(\hat{I}; H_0^1(\Omega))}^2
\end{align*}
where $\mu=\min(q+1,s).$ Scaling back to $I=(a,b)$, we have 
\begin{align*}
\|\partial_t u-\mathcal{P}_{I}^{q-1}\partial_t u\|_{L^2(I; H_0^1(\Omega))}^2=& \:\int_{a}^b \|\partial_t u (\cdot, t)-\mathcal{P}_{I}^{q-1}\partial_t u(\cdot, t)\|_{H_0^1(\Omega)}^2 \,\dd t\\
{}=&\:\left(\frac{k}{2}\right)^{-1} \int_{-1}^1 \|\partial_{\xi}\hat{u}(\cdot, \xi)-\mathcal{P}^{q-1}\partial_{\xi} u(\cdot,\xi)\|_{H_0^1(\Omega)}^2\,\dd \xi\\
{}=&\:\left(\frac{k}{2}\right)^{-1}\|\partial_{\xi}\hat{u}-\mathcal{P}^{q-1}\partial_{\xi}\hat{u}\|_{L^2(\hat{I}; H_0^1(\Omega))}^2\\
{}\leq&\: \tilde{C} \frac{1}{q^2} \frac{(q-(\mu-2))!}{(q+(\mu-2))!} \left(\frac{k}{2}\right)^{-1}|\hat{u}|_{H^{\mu}(\hat{I}; H_0^1(\Omega))}^2\\
{}=&\: \tilde{C} \frac{1}{q^2} \frac{(q-(\mu-2))!}{(q+(\mu-2))!} \left(\frac{k}{2}\right)^{-1} \int_{-1}^{1} \|\hat{u}^{(\mu)}(\cdot,\xi)\|_{H_0^1(\Omega)}^2 \,\dd \xi \\
{}=&\: \tilde{C} \frac{1}{q^2} \frac{(q-(\mu-2))!}{(q+(\mu-2))!} \left(\frac{k}{2}\right)^{-1} \int_{a}^{b} \left(\frac{k}{2}\right)^{2\mu-1} \|u^{(\mu)}(\cdot,t)\|_{H_0^1(\Omega)}^2 \,\dd t \\
{}=&\: \tilde{C} \frac{1}{q^2}\frac{(q-(\mu-2))!}{(q+(\mu-2))!} \left(\frac{k}{2}\right)^{2\mu-2} \int_{a}^{b} \|u^{(\mu)}(\cdot,t)\|_{H_0^1(\Omega)}^2 \,\dd t \\
{}=&\: \tilde{C} \frac{1}{q^2}\frac{(q-(\mu-2))!}{(q+(\mu-2))!} \left(\frac{k}{2}\right)^{2\mu-2}  |u|_{H^{\mu}(I; H_0^1(\Omega))}^2 \\
\end{align*}
where $\mu=\min(q+1,s).$ Selecting $\mu=s$, we have 
\begin{align*}
\|\partial_t u-\mathcal{P}_{I}^{q-1}\partial_t u\|_{L^2(I; H_0^1(\Omega))}^2\leq &\: \tilde{C} \frac{1}{q^2}\frac{(q-(s-2))!}{(q+(s-2))!}\left(\frac{k}{2}\right)^{2(s-1)}|u|_{H^s(I;H_0^1(\Omega))}^2\\
{}\leq &\: C\frac{k^{2(s-1)}}{q^{2(s-1)}}|u|_{H^s(I;H_0^1(\Omega))}^2
\end{align*}
where the second inequality follows from Stirling's formula as $q\to\infty.$
Therefore,
$$\|\partial_t u-\mathcal{P}_{I}^{q-1}\partial_t u\|_{L^2(I; H_0^1(\Omega))}\leq \:C\frac{k^{\mu-1}}{q^{s-1}}\|u\|_{H^s(I;H_0^1(\Omega))}$$
for $\mu=\min(q+1,s)$ since the semi-norm $|\cdot|_{H^s(I;H_0^1(\Omega))}$ is bounded by $\|\cdot\|_{H^s(I;H_0^1(\Omega))}.$
\end{proof}
\end{proposition}
As a result, the following estimates also hold.
\begin{lemma}\label{projlem}
Let $I=(a,b)$ with $k=b-a>0$. For every $u\in H^s(I;H_0^1(\Omega))$, $s\geq 2$, we have 
\begin{equation}\label{firstder}
\|\partial_t(u-\Pi_{I}^q u)\|_{L^2(I;H_0^1(\Omega))}\leq \: C \frac{k^{\mu-1}}{q^{s-1}}\|u\|_{H^s(I; H_0^1(\Omega))},
\end{equation}
\begin{equation}\label{non-derivative}
\|u-\Pi_{I}^q u\|_{L^2(I;H_0^1(\Omega))}\leq \: C \frac{k^ {\mu}}{q^{s-1}}\|u\|_{H^s(I; H_0^1(\Omega))},
\end{equation}
where $\mu=\min(q+1,s)$, $\Pi_{I}^q$ is the projection operator as defined in (\ref{projope}), and $C$ is a universal constant.
\begin{proof}
To prove (\ref{firstder}), we first note that $\partial_t(\Pi_{I}^q u)=\mathcal{P}_{I}^{q-1}\partial_{t} u$ by definition.
\begin{align*}
\|\partial_{t}(u-\Pi_{I}^q u)\|_{L^2(I;H_0^1(\Omega))}=&\:\|\partial_t u-\partial_t(\Pi_{I}^q u)\|_{L^2(I;H_0^1(\Omega))}\\
{}=&\:\|\partial_t u-\mathcal{P}_{I}^{q-1} \partial_t u\|_{L^2(I; H_0^1(\Omega))}\\
{}\leq &\: C \frac{k^{\mu-1}}{q^{s-1}}\|u\|_{H^s(I;H_0^1(\Omega))} \mbox{ (by Proposition \ref{estprop})}.
\end{align*}
For (b), we have 
\begin{align*}
\| u-\Pi_{I}^q u\|_{L^2(I; H_{0}^1(\Omega))}^2=&\: \left\Vert u(x,a)+\int_{a}^t \partial_s u(x,s)\,\dd s-u(x,a)-\int_{a}^t \mathcal{P}_{I}^{q-1}\partial_{s} u(x,s) \,\dd s\right\Vert_{L^2(I;H_0^1(\Omega))}^2\\
{}=&\:\left\Vert\int_{a}^t \partial_s u(x,s)\,\dd s-\int_{a}^t \mathcal{P}_{I}^{q-1}\partial_s u(x,s)\,\dd s\right\Vert_{L^2(I;H_0^1(\Omega))}^2\\
{}=&\: \int_{a}^b \left\Vert\int_{a}^t \partial_s u(\cdot,s)-\mathcal{P}_{I}^{q-1} \partial_{s} u(\cdot,s)\,\dd s\right\Vert_{H_0^1(\Omega)}^2\,\dd t\\
{}\leq & \:\int_{a}^b \left\Vert\left(\int_{a}^t 1\,\dd s\right)^{\frac{1}{2}}\left( \int_{a}^t |\partial_{s} u(\cdot,s)-\mathcal{P}_{I}^{q-1}\partial_s u(\cdot,s)|^2\,\dd s\right)^{\frac{1}{2}}\right\Vert_{H_0^1(\Omega)}^2 \,\dd t\\
{}\leq &\:\int_{a}^ b k\left\Vert\left( \int_{a}^t |\partial_{s} u(\cdot,s)-\mathcal{P}_{I}^{q-1}\partial_s u(\cdot,s)|^2\,\dd s\right)^{\frac{1}{2}}\right\Vert_{H_0^1(\Omega)}^2 \,\dd t\\
{}\leq &\: k^2 \left\Vert\left( \int_{a}^t |\partial_{s} u(\cdot,s)-\mathcal{P}_{I}^{q-1}\partial_{s} u(\cdot,s)|^2\,\dd s\right)^{\frac{1}{2}}\right\Vert_{H_0^1(\Omega)}^2 \\
{}=&\:k^2 \| \partial_s u-\mathcal{P}_{I}^{q-1}\partial_s u\|_{L^2(I;H_0^1(\Omega))}^2\\
\leq & \:  C^2k^2 \frac{k^{2\mu-2}}{q^{2s-2}} \|u\|_{H^s(I; H_0^1(\Omega))}^2.
\end{align*}
Thus 
\begin{align*}
\|u-\Pi_{I}^q u\|_{L^2(I; H_0^1(\Omega))}\leq  &\: C \frac{k^{\mu}}{q^{s-1}}\|u\|_{H^s(I; H_0^1(\Omega))}.
\end{align*}
\end{proof}
\end{lemma}
We can also derive the following result using Lemma \ref{projlem}.
\begin{corollary}\label{cor1}
Let $I=(a,b)$ with $k=b-a$. Then, for every $u\in H^s(I; H_0^1(\Omega))$ with $s\geq 2$, we have
\begin{equation}\label{2ndoder}
\|\partial_{tt} (u-\Pi_{I}^q u)\|_{L^2(I; H_0^1(\Omega))}\leq\: C \frac{k^{\mu-2}}{q^{s-3}}\|u\|_{H^s(I; H_0^1(\Omega))}
\end{equation}
where $\mu=\min(q+1,s)$ and $C$ is a universal constant, which may vary from line to line.
\begin{proof}
\begin{flalign*}
\|\partial_{tt}(u-\Pi_{I}^q u)\|_{L^2(I; H_0^1(\Omega))} &\leq  \|\partial_{tt} (u-\mathcal{Q}^q u)\|_{L^2(I; H_0^1(\Omega))}+\|\partial_{tt} (\mathcal{Q}^q u -\Pi_{I}^q u)\|_{L^2(I; H_0^1(\Omega))}&&\\
&\leq   C\frac{k^{\mu-2}}{q^{s-2}}\|u\|_{H^s(I; H_0^1(\Omega))}+ C_{\mathrm{inv}}\frac{q^2}{k}\|\partial_t (\mathcal{Q}^q u-\Pi_{I}^q u)\|_{L^2(I; H_0^1(\Omega))}&&\\
&= C \frac{k^{\mu-2}}{q^{s-2}}\|u\|_{H^s(I; H_0^1(\Omega))} + C_{\mathrm{inv}}\frac{q^2}{k}\|\partial_t \mathcal{Q}^q u-\mathcal{P}_{I}^{q-1} \partial_t u\|_{L^2(I; H_0^1(\Omega))}&&\\
&\leq  C\frac{k^{\mu-2}}{q^{s-2}}\|u\|_{H^s(I; H_0^1(\Omega))}+C_{\mathrm{inv}}\frac{q^2}{k}\|\partial_t(u-\mathcal{Q}^q u)\|_{L^2(I;H_0^1(\Omega))}&&\\
&\quad+ C_{\mathrm{inv}}\frac{q^2}{k}\|\partial_t u-\mathcal{P}_{I}^{q-1}\partial_t u\|_{L^2(I; H_0^1(\Omega))}&&\\
&\leq   C\frac{k^{\mu-2}}{q^{s-2}}\|u\|_{H^s(I; H_0^1(\Omega))}+C\frac{q^2}{k}\frac{k^{\mu-1}}{q^{s-1}}\| u\|_{H^s(I;H_0^1(\Omega))} &&\\
&\quad+C\frac{q^2}{k}\frac{k^{\mu-1}}{q^{s-1}}\|u\|_{H^s(I; H_0^1(\Omega))}&&\\
&\leq   C\frac{k^{\mu-2}}{q^{s-3}}\|u\|_{H^s(I; H_0^1(\Omega))},&&
\end{flalign*}
where we have used Proposition \ref{sequenceapproximate} and the inequality $\|\varphi\|_{H^1(I;H_0^1(\Omega))}\leq \:C_{\mathrm{inv}} \frac{q^2}{k}\|\varphi\|_{L^2(I;H_0^1(\Omega))}$ for some constant $C_{\mathrm{inv}}$ with $\varphi\in \mathbb{P}^{q-1}(I; H_0^1(\Omega))$ for $q\geq 1$ \cite{Sc} in the second step. Note that for the second last step, we have used the estimates from Propositions \ref{sequenceapproximate} and \ref{estprop}.
\end{proof}
\end{corollary}
\begin{remark}\label{remarkonspace}
If we change the spatial function space from $H_0^1(\Omega)$ to $L^2(\Omega)$, the same estimate follows. That is, for every $u\in H^s(I; H_0^1(\Omega))$ with $s\geq 2$, we have
\begin{equation}\label{2norder2}
\|\partial_{tt} (u-\Pi_{I}^r u)\|_{L^2(I; L^2(\Omega))}\leq\: C \frac{k^{\mu-2}}{q^{s-3}}\|u\|_{H^s(I; L^2(\Omega))}
\end{equation}
where $\mu=\min(q+1,s)$ and $C$ is a universal constant.
\end{remark}
Before we begin the proof of the main convergence theorem, we state some useful properties of the projection error first. Assume that $u\in H^s((0,T); H_0^1(\Omega))$ with $s\geq 2$ is the weak solution of (\ref{setup})--(\ref{bc}) and let $\Pi_{I}^q u\in\mathcal{V}^{\mathrm{q}}$ be the projection of $u$ such that $\Pi_{I}^q u_{\mid I_n}=\Pi_{I_n}^{q_n} u\in \mathcal{V}^{q_n}$ is defined according to (\ref{projope}) for $n=1,\ldots, N.$ Then the following properties hold for $q_n\geq 2\colon$ 
\begin{enumerate}
\item $(u-\Pi_{I}^q u)(x,t_n^{\pm})=0 \mbox{ for } n=0, 1,\ldots, N$; 
\item $\partial_t(u-\Pi_{I}^q u)(x,t_n^{-})=0 \mbox{ for } n=1,\ldots, N$;
\item $\int_{t_{n-1}}^{t_n} \left(\partial_t (u-\Pi_{I}^q),\varphi\right)\,\dd t \mbox{ for all } \varphi\in\mathbb{P}^{q_n-2}(I; H_0^1(\Omega))$.
\end{enumerate}

Now we are ready to prove the following convergence theorem.
\begin{theorem}
Let $u$ be the solution to (\ref{setup})--(\ref{bc}), and let $u_{\mathrm{DG}}\in \mathcal{V}^{\mathbf{q}}$ be the discontinuous Galerkin approximation of $u$. That is,
$$\mathcal{A}(u_{\mathrm{DG}},v)=\mathcal{F}(v)\mbox{ for all } v\in\mathcal{V}^{\mathbf{q}}.$$
If $u_{\mid I_n}\in H^{s_n}(I_n; H^2(\Omega))$ for any $n=1,\ldots, N$ with $s_n\geq 2$, then 
$$|||u-u_{\mathrm{DG}}|||\leq  C\left(\sum_{n=1}^N \frac{k_n^{2\mu_n -3}}{q_n^{2s_n-6}} \|u\|_{H^{s_{n}}( I_n; L^2(\Omega))}^2+\sum_{n=1}^N \frac{k_n^{2\mu_n}}{q_n^{2s_n-2}} \| u\|_{H^{s_n}( I_n; H^2(\Omega))}^2 \right)^{\frac{1}{2}},$$
where $\mu_n=\min(q_n+1,s_n)$ for any $n=1,\ldots, N$ and $C$ is a constant independent of $u$, $q_n$ and $k_n$.
\begin{proof} 
Denote the global error by $e:=u-u_{\mathrm{DG}}$, and decompose it as $e=e^{\pi}+e^{h}$ with
$e^{\pi}=u-\Pi_{I}^q u$ and $e^h=\Pi_{I}^q u-u_{\mathrm{DG}}.$
First note that $e^{\pi}(x,t_n^{\pm})=(u-\Pi_I^q u)(x,t_n^{\pm})=0$ for $n=0,\ldots, N$ and $\dot{e}^{\pi}(x,t_n^{-})=\partial_t(u-\Pi_{I}^r u)(x,t_n^{-})=0$ for $n=1,\ldots, N.$ Therefore,
\begin{align*} 
|||e^{\pi}|||^2=&\:\frac{1}{2}\|\dot{e}^{\pi}(t_0^{+})\|_{L^2}^2+\frac{1}{2}\sum_{n=1}^{N-1}\|[\dot{e}^{\pi}]_{n}\|_{L^2}^2+\frac{1}{2}\|\dot{e}^{\pi} (t_N^{-})\|_{L^2}^2\\
{}&+ \frac{1}{2}\|\nabla e^{\pi}(t_0^{+})\|_{L^2}^2+\frac{1}{2}\sum_{n=1}^{N-1} \| \nabla [e
^{\pi}]_n\|_{L^2}^2+\frac{1}{2}\|\nabla e^{\pi}(t_N^{-})\|_{L^2}^2+\sum_{n=1}^N \int_{t_{n-1}}^{t_n} \| \dot{e}^{\pi} \|_{L^2}^2\,\dd t\\
{}=&\:\frac{1}{2}\| \dot{e}^{\pi}(t_0^{+})\|_{L^2}^2+\frac{1}{2}\sum_{n=1}^{N-1} \|\dot{e}^{\pi}(t_{n}^{+})-\dot{e}(t_{n}^{-})\|_{L^2}^2+\sum_{n=1}^N \int_{t_{n-1}}^{t_n} \| \dot{e}^{\pi} \|_{L^2}^2 \,\dd t \\
{}=& \:\frac{1}{2}\sum_{n=1}^N \|\dot{e}^{\pi} (t_{n-1}^{+})\|_{L^2}^2 +\sum_{n=1}^N \int_{t_{n-1}}^{t_n} \| \dot{e}^{\pi}\|_{L_2}^2\,\dd t.\\
\end{align*}
By considering $\dot{e}^{\pi}(t_{n-1}^{+})=-\int_{t_{n-1}^{+}}^{t_n^{-}} \ddot{e}^{\pi}(s) \,\dd s$, we have
\begin{align*}
|||e^{\pi}|||^2=&\: \sum_{n=1}^N \int_{t_{n-1}}^{t_n}\| \dot{e}^{\pi}\|_{L^2}^2\,\dd t+\frac{1}{2}\sum_{n=1}^N\| -\int_{t_{n-1}}^{t_n} \ddot{e}^{\pi}(s)\,\dd s\|_{L^2}^2\, \dd t \\
\leq &\: \sum_{n=1}^N \left( \|\dot{e}^{\pi} \|_{L^2(I_n; L^2(\Omega))}^2+\frac{k_n}{2} \|\ddot{e}^{\pi}\|_{L^2(I_n; L^2(\Omega))}^2 \right)\mbox{ (by H\"{o}lder's inequality)}\\
{}=&\: \sum_{n=1}^N \left( \| \partial_t (u-\Pi_{I_n}^{q_n} u)\|_{L^2(I_n;L^2(\Omega))}^2+\frac{k_n}{2} \|\partial_{tt} (u-\Pi_{I_n}^{q_n} u) \|_{L^2(I_n; L^2(\Omega))}^2 \right)\\
{}\leq&\: C\sum_{n=1}^N \frac{k_n^{2\mu_n-2}}{q_n^{2s_n-2}}\|u\|_{H^s(I_n; L^2(\Omega))}^2+C\sum_{n=1}^N \frac{k_n^{2\mu_n-3}}{q^{2s_n-6}}\|u\|_{H^s(I_n; L^2(\Omega))}^2 \mbox{ (by (\ref{firstder}) and  (\ref{2norder2})) }\\
{}\leq &\: C \sum_{n=1}^N \frac{k_n^{2\mu_n-3}}{q_n^{2s_n-6}}\|u\|_{H^s(I_n; L^2(\Omega))}^2,
\end{align*}
where $\mu_n=\min (q_n+1, s_n)$ for any $n=1,\ldots, N$ and $C$ is a universal constant, which may vary from line to line.

By Galerkin orthogonality, we obtain
$\mathcal{A}(e^{\pi}+e^h, e^h)=\mathcal{A}(u-u_{\mathrm{DG}}, e^h)=0$ as $e^h=\Pi_{I}^q u-u_{\mathrm{DG}}\in \mathcal{V}^{\mathbf{q}}.$ Then 
\begin{align*}
|||e^h|||^2=&\:\mathcal{A}(e^h, e^h)\\
{}=&\: -\mathcal{A} (e^{\pi}, e^h)\\
{}=&\:-\sum_{n=1}^N \int_{t_{n-1}}^{t_n} \left\langle\ddot{e}^{\pi}, \dot{e}^h\right\rangle \,\dd t -\sum_{n=1}^N \int_{t_{n-1}}^{t_n} (\dot{e}^{\pi},  \dot{e}^h)_{L^2} \,\dd t-\sum_{n=1}^{N} \int_{t_{n-1}}^{t_n} (\nabla e^{\pi}, \nabla \dot{e}^h)_{L^2}\,\dd t\\
{}&- \sum_{n=1}^{N-1} ([\dot{e}^{\pi} ]_n, \dot{e}^h(t_n^{+}))_{L^2}-\sum_{n=1}^{N-1} (\nabla [e^{\pi}]_n, \nabla e^h (t_n^{+}) )_{L^2}-(\dot{e}^{\pi}(t_0^{+}), \dot{e}^h(t_0^{+}))_{L^2}\\
{}&-(\nabla e^{\pi}(t_0^{+}), \nabla e^h(t_0^{+})).
\end{align*}
Integrating by parts the term $\left\langle\ddot{e}^{\pi}, \dot{e}^h\right\rangle$ and rearranging the addends, we have 
\begin{align*}
|||e^h|||^2=\:&\sum_{n=1}^N \int_{t_{n-1}}^{t_n} (\dot{e}^{\pi}, \ddot{e}^h)_{L^2} \,\dd t -\sum_{n=1}^N \int_{t_{n-1}}^{t_n} (\dot{e}^{\pi}, \dot{e}^h)_{L^2} \,\dd t-\sum_{n=1}^{N} \int_{t_{n-1}}^{t_n} (\nabla e^{\pi}, \nabla \dot{e}^h)_{L^2}\,\dd t\\
{}&+ \sum_{n=1}^{N-1} (\dot{e}^{\pi} (t_n^{-}), [\dot{e}^h]_n)_{L^2}-\sum_{n=1}^{N-1} (\nabla [e^{\pi}]_n,\nabla e^h (t_n^{+}) )_{L^2}-(\dot{e}^{\pi}(t_N^{-}), \dot{e}^h(t_N^{-}))_{L^2}\\
{}&-(\nabla e^{\pi}(t_0^{+}), \nabla e^h(t_0^{+}))_{L^2}.
\end{align*}
Note that the first term vanishes by the orthogonality condition that $\int_{t_{n-1}}^{t_n} \left(\dot{e}^{\pi}, \varphi\right)_{L^2}\,\dd t=0$ for all $\varphi\in \mathbb{P}^{q_n-2}(I_n;H_0^1(\Omega))$, and the terms on the second and third lines vanish by the properties of the projection error. Integrating by parts the term $\sum_{n=1}^N \int_{t_{n-1}}^{t_n} (\nabla e^{\pi}, \nabla \dot{e}^h )_{L^2}\,\dd t$ in the $x$ variable, we have
\begin{align*}
|||e^h|||^2=&\:-\sum_{n=1}^N \int_{t_{n-1}}^{t_n} (\dot{e}^{\pi}, \dot{e}^h)_{L^2} \,\dd t+\sum_{n=1}^N \int_{t_{n-1}}^{t_n} (\Delta e^{\pi},  \dot{e}^h )_{L^2}\,\dd t\\
{}\leq &\: \sum_{n=1}^N \int_{t_{n-1}}^{t_n} \int_{\Omega} | \dot{e}^{\pi}|| \dot{e}^h |\,\dd x\, \dd t+\sum_{n=1}^N \int_{t_{n-1}}^{t_n} \int_{\Omega} |\Delta e^{\pi}||\dot{e}^h |\,\dd x \,\dd t \\
{}\leq &\:\sum_{n=1}^N \| \dot{e}^{\pi} \|_{L^2(I_n; L^2(\Omega))}\|\dot{e}^h\|_{L^2(I_n; L^2(\Omega))}+\sum_{n=1}^N \|\Delta e^{\pi} \|_{L^2(I_n; L^2(\Omega))}\|\dot{e}^h\|_{L^2(I_n; L^2(\Omega))}\\
{}\leq &\: \sum_{n=1}^N \|\dot{e}^{\pi}\|_{L^2(I_n; L^2(\Omega))}^2+\frac{1}{4}\|\dot{e}^{h}\|_{L^2(I_n; L^2(\Omega))}^2+ \sum_{n=1}^N \|\Delta e^{\pi}\|_{L^2(I_n; L^2(\Omega))}^2+\frac{1}{4}\|\dot{e}^{h}\|_{L^2(I_n; L^2(\Omega))}^2
\end{align*}
by Young's inequality. Note that $\|\dot{e}^h\|_{L^2(I_n;L^2(\Omega))}\leq |||e^h|||$, so we can absorb the $\frac{1}{2}\|\dot{e}^h\|_{L^2(I_n;L^2(\Omega))}^2$ term into the left-hand side of the inequality to get
\begin{align*}
|||e^h|||^2\leq & \:2\sum_{n=1}^N \left( \|\dot{e}^{\pi}\|_{L^2(I_n; L^2(\Omega))}^2+\|\Delta e^{\pi}\|_{L^2(I_n; L^2(\Omega))}^2\right)\\
{}\leq & \:C\sum_{n=1}^N \frac{k_n^{2\mu_n -2}}{q_n^{2s_n-2}} \|u\|_{H^{s_{n}}( I_n; L^2(\Omega))}^2+\:C\sum_{n=1}^N \frac{k_n^{2\mu_n}}{q_n^{2s_n-2}} \|\Delta u\|_{H^{s_{n}}( I_n; L^2(\Omega))}^2 \mbox{    (by Lemma \ref{projlem})},
\end{align*}
where $\mu_n=\min(q_n+1, s_n)$ for any $n=1,\ldots, N.$ By the definition of the $H^2$--norm, we know that $\|\Delta u\|_{L^2(\Omega)}\leq \|u\|_{H^2(\Omega)}$.
Therefore,
\begin{align*}
|||e||| &\leq  |||e^{\pi}|||+|||e^h|||\\
&\leq   \left(C\sum_{n=1}^N \frac{k_n^{2\mu_n-3}}{q_n^{2s_n-6}} \|u\|_{H^s(I_n; L^2(\Omega))}^2\right)^{\frac{1}{2}}\\
& \quad+ \left(C\sum_{n=1}^N \frac{k_n^{2\mu_n -2}}{q_n^{2s_n-2}} \|u\|_{H^{s_{n}}( I_n; L^2(\Omega))}^2+C\sum_{n=1}^N \frac{k_n^{2\mu_n}}{q_n^{2s_n-2}} \| u\|_{H^{s_{n}}( I_n; H^2(\Omega))}^2 \right)^{\frac{1}{2}}\\
&\leq  C\left(\sum_{n=1}^N \frac{k_n^{2\mu_n -3}}{q_n^{2s_n-6}} \|u\|_{H^{s_{n}}( I_n; L^2(\Omega))}^2+\sum_{n=1}^N \frac{k_n^{2\mu_n}}{q_n^{2s_n-2}} \| u\|_{H^{s_{n}}( I_n; H^2(\Omega))}^2 \right)^{\frac{1}{2}},
\end{align*}
where $C$ is a universal constant, which may vary from line to line.
\end{proof}
\end{theorem}
\begin{remark}\label{uniform_grid}
If we use uniform time intervals $k_n= k$, and uniform polynomial degrees $q_n=q\geq 2, s_n =s\geq 2$ for $n=1,\ldots, N$, then the error bound becomes
$$|||u-u_{\mathrm{DG}}|||\leq C\left(\frac{k^{2\mu -3}}{q^{2s-6}} \|u\|_{H^{s}( 0,T; L^2(\Omega))}^2+ \frac{k^{2\mu}}{q^{2s-2}} \| u\|_{H^{s_{n}}(0,T; H^2(\Omega))}^2 \right)^{\frac{1}{2}}\leq C\frac{k^{\mu-\frac{3}{2}}}{q^{s-3}}\| u\|_{H^{s}(0,T; H^2(\Omega))}$$
where $\mu=\min(q+1,s)$ and $C$ is a constant independent of $u$, $q$ and $k$.
\end{remark}
\section{Fully discrete numerical scheme}\label{fullydiscrete}
In this section, we shall construct a fully discrete scheme for the approximation of the solution of (\ref{setup})--(\ref{bc}). This numerical scheme combines a $hp$--DGFEM in the time direction with an $H^1(\Omega)$-conforming finite element approximation in the spatial variables. It is well-known that discontinuous Galerkin finite element methods offer certain advantages over standard continuous Galerkin methods when applied to the spatial discretization of the acoustic wave equation \cite{AMM}. For instance, the mass matrix has a desirable block diagonal structure, which gives a more efficient computation when using an explicit time scheme. However, for the sake of having a neat and concise convergence analysis, we stick to a continuous Galerkin method in the spatial direction for this paper; applying the DGFEM in both the time and spatial directions may be considered for our future work.
\subsection{Construction of the fully discrete scheme}
For the spatial discretization parameter $h\in(0,1)$, we define $\mathcal{V}_h$ to be a given family of finite element subspaces of  $H_0^1(\Omega)$ with polynomial degree $p\geq 1$ such that, 
\begin{equation}\label{inte_prop}
\inf_{v\in \mathcal{V}_h} \{\|u-v\|_{L^2}+h\|u-v\|_{H^1}\}\leq C h^{r+1}\|u\|_{H^{r+1}}, \quad 1\leq r\leq \min(m,p), 
\end{equation}
for $u\in H^m(\Omega)\cap H_0^1(\Omega).$

Now we introduce the space-time finite element space by 
$$\mathcal{V}_{kh}^{q_n}:=\{u\colon [0, \infty)\to \mathcal{V}_h;\: u\mid_{I_n} =\sum_{j=1}^{q_n} u_j t^j,\:u_j\in\mathcal{V}_h\},$$
with $q_n\geq 2$ for each $1\leq n\leq N$. For $\mathbf{q}:=[q_1, \ldots, q_N]^T\in\mathbb{N}^N$, we then define the space
$$\mathcal{V}_{kh}^{\mathbf{q}}:=\{u\colon [0, \infty)\to \mathcal{V}_{h};\: u\mid_{I_n}\in\mathcal{V}_{kh}^{q_n} \: \mbox{ for } n=1,\ldots,N\}.$$
Following the same procedure as in section \ref{semi-discrete}, we can now define the fully discrete discontinuous-in-time scheme. We seek a solution $u_{\mathrm{DG}}\in\mathcal{V}_{kh}^{\mathbf{q}}$ such that  
\begin{equation}\label{discrete_formulation}
\mathcal{A}(u_{\mathrm{DG}},v)=\mathcal{\tilde{F}}(v) \quad \text{ for all } v\in \mathcal{V}_{kh}^{\mathbf{q}},
\end{equation}
where the bilinear form $\mathcal{A}(\cdot, \cdot)\colon \mathcal{H}\times\mathcal{H}\to\mathbb{R}$ is defined by (\ref{case1blinear}) and $\mathcal{\tilde{F}}$ is a modified version of $\mathcal{F}$ in (\ref{LF}) defined as 
\begin{equation}\label{newLF}
\mathcal{\tilde{F}}(v):=\sum_{n=1}^N\int_{t_{n-1}}^{t_n}\left(f(t),\dot{v}(t)\right)_{L^2} \dd t+\left(\mathcal{P}_h u_1, \dot{v}(t_0^{+})\right)_{L^2}+\left(\nabla \mathcal{P}_h u_0, \nabla v(t_0^{+})\right)_{L^2},
\end{equation}
where $\mathcal{P}_h\colon H_0^1(\Omega)\to \mathcal{V}_h$ is the Ritz projection such that 
\begin{equation}\label{Ritz}
\left(\nabla \mathcal{P}_h u,\nabla \varphi\right)_{L^2}=\left(\nabla  u, \nabla\varphi\right)_{L^2} \quad \mbox{ for all } \varphi\in\mathcal{V}_h.
\end{equation}
The existence and uniqueness of the fully discrete solution $u_{\mathrm{DG}}\in\mathcal{V}_{kh}^{\mathbf{q}}$ of (\ref{discrete_formulation}) follow from the following proposition.
\begin{proposition}\label{existence}
For each $1\leq n\leq N$, the local problem
\begin{equation}\label{localprob}
\begin{split}
&\int_{t_{n-1}}^{t_n} \left\langle \ddot{u}, \dot{v}\right\rangle \,\dd t+\int_{t_{n-1}}^{t_n} (\dot{u}, \dot{v})_{L^2} \,\dd t+\int_{t_{n-1}}^{t_n} (\nabla u, \nabla \dot{v} )_{L^2}\,\dd t +(\dot{u}_{n-1}^{+}, \dot{v}_{n-1}^{+})_{L^2}+(\nabla u_{n-1}^{+}, \nabla v_{n-1}^{+})_{L^2}\\
&= \int_{t_{n-1}}^{t_n} (f,\dot{v})_{L^2} \,\dd t +(\dot{u
}_{n-1}^{-}, \dot{v}_{n-1}^{+})_{L^2}+(\nabla u_{n-1}^{-}, \nabla v_{n-1}^{+})_{L^2}
\end{split}
\end{equation}
admits a unique solution $u\in\mathcal{V}_{kh}^{q_n}$ on $I_n$ provided that $u_{n-1}$ and $f\mid_{I_n}$ are given.
\begin{proof}
We first note that to show uniqueness, it suffices to see that the corresponding homogeneous equation 
\begin{equation}\label{homoeq}
\begin{split}
{}&\int_{t_{n-1}}^{t_n} \left\langle \ddot{u}, \dot{v}\right\rangle \,\dd t +\int_{t_{n-1}}^{t_n} (\dot{u},\dot{v})_{L^2} \,\dd t+\int_{t_{n-1}}^{t_n} (\nabla u,\nabla \dot{v})_{L^2}\,\dd t+(\dot{u}_{n-1}^{+}, \dot{v}_{n-1}^{+})_{L^2}+(\nabla u_{n-1}^{+}, \nabla v_{n-1}^{+})_{L^2}=0
\end{split}
\end{equation}
for all $v\in \mathcal{V}_{kh}^{q_n}$ only has the trivial solution $u\equiv 0.$ For this purpose, we assume that $u$ is a solution and choose $v=u$ on $I_n$, then we have 
\begin{equation}
\|\dot{u}_n^{-}\|_{L^2}^2-\|\dot{u}_{n-1}^{+}\|_{L^2}^2 +2\int_{t_{n-1}}^{t_n} \|\dot{u}\|_{L^2}^2 \,\dd t+\|\nabla u_n^{-}\|_{L^2}^2-\|\nabla u_{n-1}^{+}\|_{L^2}^2+ 2\|\dot{u}_{n-1}^{+}\|_{L^2}^2+2\|\nabla u_{n-1}^{+}\|_{L^2}^2=0.
\end{equation}
That is, 
\begin{equation}
\|\dot{u}_{n}^{-}\|_{L^2}^2+\|\dot{u}_{n-1}^{+}\|_{L^2}^2+2\int_{t_{n-1}}^{t_n} \|\dot{u}\|_{L^2}^2\,\dd t+\|\nabla u_n^{-}\|_{L^2}^2+\|\nabla u_{n-1}^{+}\|_{L^2}^2=0.
\end{equation} Note that $\|u_{n-1}^{+}\|_{H_0^1}=\|\nabla u_{n-1}^{+}\|_{L^2}=0$ implies that $u_{n-1}^{+}\equiv 0.$  Then we have
$$\begin{cases}
\dot{u}(t)=0 & \mbox{ for } t\in I_n,\\
u_{n-1}^{+}=0.
\end{cases}$$
This implies that $u(t)\equiv 0$ on $I_n.$ The existence of a solution to the local problem (\ref{localprob}) follows from the uniqueness since this is a finite dimensional problem.
\end{proof}
\end{proposition}
\subsection{Convergence analysis}
\begin{theorem}\label{convergence_thm}
Let $u$ be the solution to (\ref{setup})--(\ref{bc}) such that $u_{\mid I_n}\in H^{s_n}(I_n; H^{m+1}(\Omega))$ for any $n=1,\ldots, N$ with $s_n\geq 2$, and let $u_{\mathrm{DG}}\in \mathcal{V}_{kh}^{\mathbf{q}}$ be the discontinuous Galerkin approximation of $u$. That is,
$$\mathcal{A}(u_{\mathrm{DG}},v)=\tilde{\mathcal{F}}(v)\mbox{ for all } v\in\mathcal{V}_{kh}^{\mathbf{q}}.$$
Assume that $1\leq r\leq \min (p,m)$, $s_n\geq q_n+1$ for each $n=1,\ldots, N$; then we have
\begin{align}\label{error_in_energy}
|||u-u_{\mathrm{DG}}||| &\leq C\bigg\{\sum_{n=1}^N \frac{k_n^{2q_n-1}}{q_n^{2s_n-6}} \|u\|_{H^{s_n}(I_n; H^{r+1}(\Omega))}^2+ h^{2r+2}\|u\|_{H^{s_n}(0,T; H^{r+1}(\Omega))}^2\bigg\}^{\frac{1}{2}} \nonumber\\
&\quad+ \bigg\{h^{2r+2} \left(\|\dot{u}(t_0^{+})\|_{H^{r+1}}^2+\|\dot{u}(t_N^{-})\|_{H^{r+1}}^2\right) + h^{2r} \left( \| u(t_0^{+})\|_{H^{r+1}}^2 +\| u(t_N^{-})\|_{H^{r+1}}^2\right)\bigg\}^{\frac{1}{2}},
\end{align}
where $C$ is a positive constant, which may vary from line to line.
\end{theorem}
\begin{proof}
Let $\Pi_{I_n}^{q_n}$ denote the modified $L^2$-projector defined by Definition \ref{proj} on $I_n$ for each $n=1,\ldots,N$. We decompose the error as 
\begin{align*}
u_{\mathrm{DG}}(t)-u(t) &=  \left(u_{\mathrm{DG}}(t)-\Pi_{I_n}^{q_n}\mathcal{P}_hu(t)\right)+\left(\Pi_{I_n}^{q_n}\mathcal{P}_hu(t)-\mathcal{P}_h u(t)\right)+\left(\mathcal{P}_hu(t)-u(t)\right)\\
&=: \theta(t)+\rho_1(t)+\rho_2(t),
\end{align*}
for $t\in I_n$, with $n=1,\ldots,N.$ For simplicity of notation, we define 
$(\Pi_{I_n}^{q_n} w)^{(j)}:=\frac{\dd^{j}\Pi_{I_n}^{q_n} w}{\dd t^j}$ and $u^{(j)}:=\frac{\dd^j u}{\dd t^j}$, for $j=0, 1$ and $2$.  By Definition \ref{proj}, Lemma \ref{projlemma} and Corollary \ref{cor1}, we have 
\begin{align}
\int_{t_{n-1}}^{t_n} \|\rho_1^{(j)}(t)\|_{L^2}^2 \:\dd t = & \int_{t_{n-1}}^{t_n} \| (\Pi_{I_n}^{q_n} \mathcal{P}_h u)^{(j)}(t)-(\mathcal{P}_h u)^{(j)}(t) \|_{L^2}^2 \:\dd t\nonumber\\
{}\leq &\:C\frac{k_n^{2(\mu-j)}}{q_n^{2(s_n-1)}} \|\mathcal{P}_h u\|_{H^{s_n}(I_n; L^2(\Omega))}^2 \nonumber\\
{}\leq &\:C\frac{k_n^{2(\mu-j)}}{q_n^{2(s_n-1)}} \| u\|_{H^{s_n}(I_n; L^2(\Omega))}^2 \quad\text{ for } j=0,1,
\end{align}
and
\begin{align}
\int_{t_{n-1}}^{t_n} \|\rho_1^{(2)}(t)\|_{L^2}^2 \:\dd t = & \int_{t_{n-1}}^{t_n} \| (\Pi_{I_n}^{q_n  } \mathcal{P}_h u)^{(2)}(t)-(\mathcal{P}_h u)^{(2)}(t) \|_{L^2}^2 \dd t \leq \: C \frac{k_n^{2(\mu-2)}}{q_n^{2(s_n-3)}}  \| u\|_{H^{s_n}(I_n; L^2(\Omega))}^2,
\end{align}
where $\mu=\min(s_n,q_n+1)$ and $C$ is a universal constant, which may vary from line to line. If we assume that the solution $u$ of (\ref{setup})--(\ref{bc}) is sufficiently smooth (i.e. $s_n>q_n+1$), then we can write
\begin{equation}\label{rho1}
\int_{t_{n-1}}^{t_n} \|\rho_1^{(j)}(t)\|_{L^2}^2 \:\dd t\leq C  \frac{k_n^{2(q_n+1-j)}}{q_n^{2(s_n-1)}} \| u\|_{H^{s_n}(I_n; L^2(\Omega))}^2,
\end{equation}
for $j=0,1$, and 
\begin{equation}
\int_{t_{n-1}}^{t_n} \|\rho_1^{(2)}(t)\|_{L^2}^2 \:\dd t  \leq \: C \frac{k_n^{2q_n-2}}{q_n^{2(s_n-3)}} \| u\|_{H^{s_n}(I_n; L^2(\Omega))}^2.
\end{equation}
By C\'{e}a's lemma and the interpolation property (\ref{inte_prop}), we know that
\begin{equation}
\|\nabla(\mathcal{P}_h u-u)\|_{L^2}\leq \min_{v\in \mathcal{V}_h}\|\nabla u-\nabla v\|_{L^2}\leq C h^{r}\|u\|_{H^{r+1}}.
\end{equation}
Applying a duality argument (cf. Theorem 5.4.8 in \cite{BS}), we have
\begin{equation}
\|\mathcal{P}_h u-u\|_{L^2}\leq h\|\nabla (\mathcal{P}_h u-u)\|_{L^2}\leq C h^{r+1}\|u\|_{H^{r+1}}.
\end{equation}
Thus, we can approximate the terms involving $\rho_2$ by the following:
\begin{align}
\int_{t_{n-1}}^{t_n} \|\rho_2^{(j)}(t)\|_{L^2}^2 \:\dd t & = \int_{t_{n-1}}^{t_n}\| \mathcal{P}_h u^{(j)}(t)-u^{(j)}(t)\|_{L^2}^2 \:\dd t \nonumber\\
&\leq C \int_{t_{n-1}}^{t_n} h^{2r+2} \|u^{j}(t)\|_{H^{r+1}}^2 \:\dd t \nonumber\\
&= C h^{2r+2} \|u\|_{H^j(I_n;H^{r+1}(\Omega))}^2 \quad\text{ for } j=0,1,2.
\end{align}
Recall that the fully discrete scheme is 
\begin{flalign}\label{discrete_eq}
&\sum_{n=1}^N\int_{t_{n-1}}^{t_n}\: \left\langle\ddot{u}_{\mathrm{DG}}(t),\dot{v}(t)\right\rangle \,\dd t +\sum_{n=1}^N\int_{t_{n-1}}^{t_n}\:\left(\dot{u}_{\mathrm{DG}}(t),\dot{v}(t)\right)_{L^2}\, \dd t+\sum_{n=1}^N\int_{t_{n-1}}^{t_n}\: \left(\nabla u_{\mathrm{DG}}(t), \nabla \dot{v}(t)\right)_{L^2} \dd t\nonumber &&\\
&\quad+\sum_{n=1}^{N-1}([\dot{u}_{\mathrm{DG}}]_{n}, \dot{v}(t_n^{+}))_{L^2}+\sum_{n=1}^{N-1}([\nabla u_{\mathrm{DG}}]_{n}, \nabla v(t_n^{+}))_{L^2} +(\dot{u}_{\mathrm{DG}}(t_0^{+}), \dot{v}(t_0^{+}))_{L^2}+(\nabla u_{\mathrm{DG}}(t_0^{+}),\nabla v(t_0^{+}))_{L^2}\nonumber &&\\
&= \sum_{n=1}^N\int_{t_{n-1}}^{t_n}\left(f(t),\dot{v}(t)\right)_{L^2} \dd t+\left(\mathcal{P}_h u_1, \dot{v}(t_0^{+})\right)_{L^2}+\left(\nabla \mathcal{P}_h u_0, \nabla v(t_0^{+})\right)_{L^2}&&
\end{flalign}
for all $v\in \mathcal{V}_{kh}^{\mathbf{q}}.$ The variational form of the original problem is written as 
\begin{flalign}\label{variational_eq}
&\sum_{n=1}^N\int_{t_{n-1}}^{t_n}\: \left\langle\ddot{u}(t),\dot{v}(t)\right\rangle \,\dd t +\sum_{n=1}^N\int_{t_{n-1}}^{t_n}\:\left(\dot{u}(t),\dot{v}(t)\right)_{L^2}\, \dd t+\sum_{n=1}^N\int_{t_{n-1}}^{t_n}\: \left(\nabla u(t), \nabla \dot{v}(t)\right)_{L^2} \dd t\nonumber&&\\
&\quad+\sum_{n=1}^{N-1}([\dot{u}]_{n}, \dot{v}(t_n^{+}))_{L^2}+\sum_{n=1}^{N-1}([\nabla u]_{n}, \nabla v(t_n^{+}))_{L^2} +(\dot{u}(t_0^{+}), \dot{v}(t_0^{+}))_{L^2}+(\nabla u(t_0^{+}),\nabla v(t_0^{+}))_{L^2}\nonumber &&\\
&= \sum_{n=1}^N\int_{t_{n-1}}^{t_n}\left(f(t),\dot{v}(t)\right)_{L^2} \dd t+\left( u_1, \dot{v}(t_0^{+})\right)_{L^2}+\left(\nabla u_0, \nabla v(t_0^{+})\right)_{L^2}&&
\end{flalign}
for all $v\in \mathcal{V}_{kh}^{\mathbf{q}}.$
Subtracting (\ref{variational_eq}) from (\ref{discrete_eq}), we have 
\begin{flalign}\label{error_eq}
&\sum_{n=1}^N\int_{t_{n-1}}^{t_n}\: \left\langle\ddot{\theta}(t)+\ddot{\rho}_1(t)+\ddot{\rho}_2(t),\dot{v}(t)\right\rangle \,\dd t +\sum_{n=1}^N\int_{t_{n-1}}^{t_n}\:\left(\dot{\theta}(t)+\dot{\rho}_1(t)+\dot{\rho}_2(t),\dot{v}(t)\right)_{L^2}\, \dd t \nonumber &&\\
&\quad+\sum_{n=1}^N\int_{t_{n-1}}^{t_n}\: \left(\nabla \theta(t) +\nabla \rho_1(t)+\nabla \rho_2(t), \nabla \dot{v}(t)\right)_{L^2} \dd t +\sum_{n=1}^{N-1}([\dot{\theta}+\dot{\rho}_1+\dot{\rho}_2]_{n}, \dot{v}(t_n^{+}))_{L^2} \nonumber &&\\
&\quad+\sum_{n=1}^{N-1}( [\nabla\theta +\nabla \rho_1+\nabla \rho_2]_{n}, \nabla v(t_{n}^{+}))_{L^2}+(\dot{u}_{\mathrm{DG}}(t_0^{+})-\dot{u}(t_0^{+}), \dot{v}(t_0^{+}))_{L^2}\nonumber &&\\
&\quad+(\nabla u_{\mathrm{DG}}(t_0^{+})-\nabla u(t_0^{+}),\nabla v(t_0^{+}))_{L^2}\nonumber &&\\
&= \left( \mathcal{P}_h u_1- u_1, \dot{v}(t_0^{+})\right)_{L^2}+\left(\nabla\mathcal{P}_h u_0-\nabla u_0, \nabla v(t_0^{+})\right)_{L^2}.&&
\end{flalign}
By taking $v=\theta$ in (\ref{error_eq}), we have 

\begin{flalign}\label{theta_err1}
&\sum_{n=1}^N\int_{t_{n-1}}^{t_n}\: \left\langle\ddot{\theta}(t),\dot{\theta}(t)\right\rangle \,\dd t +\sum_{n=1}^N\int_{t_{n-1}}^{t_n}\:\left(\dot{\theta}(t),\dot{\theta}(t)\right)_{L^2}\, \dd t+\sum_{n=1}^N\int_{t_{n-1}}^{t_n}\: \left(\nabla \theta(t), \nabla \dot{\theta}(t)\right)_{L^2} \dd t \nonumber &&\\
&\quad+\sum_{n=1}^{N-1}([\dot{\theta}]_{n}, \dot{\theta}(t_n^{+}))_{L^2}+\sum_{n=1}^{N-1}(\nabla [\theta]_{n}, \nabla \theta(t_{n}^{+}))_{L^2}+(\dot{\theta}(t_0^{+}), \dot{\theta}(t_0^{+}))_{L^2}+(\nabla\theta(t_0^{+}),\nabla \theta(t_0^{+}))_{L^2}\nonumber &&\\
&= -\sum_{n=1}^N\int_{t_{n-1}}^{t_n}\: \left\langle\ddot{\rho}_1(t),\dot{\theta}(t)\right\rangle \,\dd t-\sum_{n=1}^N\int_{t_{n-1}}^{t_n}\:\left(\dot{\rho}_1(t),\dot{\theta}(t)\right)_{L^2}\, \dd t-\sum_{n=1}^N\int_{t_{n-1}}^{t_n}\: \left(\nabla \rho_1(t), \nabla \dot{\theta}(t)\right)_{L^2} \dd t\nonumber &&\\
&\quad-\sum_{n=1}^{N-1}\left([\dot \rho_1]_n, \dot\theta(t_n^{+})\right)_{L^2}-\sum_{n=1}^{N-1}\left([\nabla\rho_1]_n,\nabla\theta(t_n^{+})\right)_{L^2}-\left(\dot{\rho_1}(t_0^{+}), \dot{\theta}(t_0^{+})\right)_{L^2}-\left(\nabla\rho_1(t_0^{+}), \nabla\theta(t_0^{+})\right)_{L^2}\nonumber &&\\
&\quad -\sum_{n=1}^N\int_{t_{n-1}}^{t_n}\: \left\langle\ddot{\rho}_2(t),\dot{\theta}(t)\right\rangle \,\dd t-\sum_{n=1}^N\int_{t_{n-1}}^{t_n}\:\left(\dot{\rho}_2(t),\dot{\theta}(t)\right)_{L^2}\, \dd t-\sum_{n=1}^N\int_{t_{n-1}}^{t_n}\: \left(\nabla \rho_2(t), \nabla \dot{\theta}(t)\right)_{L^2} \dd t\nonumber &&\\
&\quad-\sum_{n=1}^{N-1}\left([\dot \rho_2]_n, \dot\theta(t_n^{+})\right)_{L^2}-\sum_{n=1}^{N-1}\left([\nabla\rho_2]_n,\nabla\theta(t_n^{+})\right)_{L^2}-\left(\dot{\rho_2}(t_0^{+}), \dot{\theta}(t_0^{+})\right)_{L^2}-\left(\nabla\rho_2(t_0^{+}), \nabla\theta(t_0^{+})\right)_{L^2}\nonumber &&\\
& \quad+\left( \mathcal{P}_h u_1- u_1, \dot{\theta}(t_0^{+})\right)_{L^2}+\left(\nabla\mathcal{P}_h u_0-\nabla u_0, \nabla \theta(t_0^{+})\right)_{L^2}.&&
\end{flalign}
Since $\rho_2(t):=\mathcal{P}_hu(t)-u(t)$,  both $\rho_2(t)$ and $\dot{\rho_2}(t)$ are continuous in $t$. By the property of the Ritz operator, we also have that 
$$\left(\nabla \rho_2(t), \nabla \dot{\theta}(t)\right)_{L^2}=0 \quad \mbox{ for all } t\in I_n, \: n=1,\ldots,N.$$ 
Thus, we can write (\ref{theta_err1}) as follows:
\begin{flalign}\label{theta_err2}
&\sum_{n=1}^N\int_{t_{n-1}}^{t_n}\: \left\langle\ddot{\theta}(t),\dot{\theta}(t)\right\rangle \,\dd t +\sum_{n=1}^N\int_{t_{n-1}}^{t_n}\:\left(\dot{\theta}(t),\dot{\theta}(t)\right)_{L^2}\, \dd t+\sum_{n=1}^N\int_{t_{n-1}}^{t_n}\: \left(\nabla \theta(t), \nabla \dot{\theta}(t)\right)_{L^2} \dd t \nonumber &&\\
&\quad+\sum_{n=1}^{N-1}([\dot{\theta}]_{n}, \dot{\theta}(t_n^{+}))_{L^2}+\sum_{n=1}^{N-1}(\nabla [\theta]_{n}, \nabla \theta(t_{n}^{+}))_{L^2}+(\dot{\theta}(t_0^{+}), \dot{\theta}(t_0^{+}))_{L^2}+(\nabla\theta(t_0^{+}),\nabla \theta(t_0^{+}))_{L^2}\nonumber &&\\
&= -\sum_{n=1}^N\int_{t_{n-1}}^{t_n}\: \left\langle\ddot{\rho}_1(t),\dot{\theta}(t)\right\rangle \,\dd t-\sum_{n=1}^N\int_{t_{n-1}}^{t_n}\:\left(\dot{\rho}_1(t),\dot{\theta}(t)\right)_{L^2}\, \dd t-\sum_{n=1}^N\int_{t_{n-1}}^{t_n}\: \left(\nabla \rho_1(t), \nabla \dot{\theta}(t)\right)_{L^2} \dd t\nonumber &&\\
&\quad-\sum_{n=1}^{N-1}\left([\dot \rho_1]_n, \dot\theta(t_n^{+})\right)_{L^2}-\sum_{n=1}^{N-1}\left([\nabla\rho_1]_n,\nabla\theta(t_n^{+})\right)_{L^2}-\left(\dot{\rho_1}(t_0^{+}), \dot{\theta}(t_0^{+})\right)_{L^2}-\left(\nabla\rho_1(t_0^{+}), \nabla\theta(t_0^{+})\right)_{L^2}\nonumber &&\\
&\quad -\sum_{n=1}^N\int_{t_{n-1}}^{t_n}\: \left\langle\ddot{\rho}_2(t),\dot{\theta}(t)\right\rangle \,\dd t-\sum_{n=1}^N\int_{t_{n-1}}^{t_n}\:\left(\dot{\rho}_2(t),\dot{\theta}(t)\right)_{L^2}\, \dd t. &&
\end{flalign}
Note that 
\begin{flalign}\label{integrationbyparts}
\int_{t_{n-1}}^{t_n} \left\langle\ddot \rho_1(t), \dot\theta(t)\right\rangle \:\dd t = & \int_{t_{n-1}}^{t_n} \frac{\dd }{\dd t} \left(\dot \rho_1(t), \dot\theta(t)\right) \:\dd t -\int_{t_{n-1}}^{t_n} \left(\dot{\rho}_1(t), \ddot{\theta}(t)\right)_{L^2} \:\dd t \nonumber && \\
=& \:\left(\dot{\rho_1}(t_n^{-}),\dot{\theta}(t_n^{-})\right)_{L^2}-\left(\dot{\rho_1}(t_{n-1}^{+}), \dot{\theta}(t_{n-1}^{+})\right)_{L^2},&&
\end{flalign}
where we have used (d) of Lemma \ref{projproperties}. Substituting (\ref{integrationbyparts}) into (\ref{theta_err2}) and rearranging the addends yield

\begin{flalign}\label{theta_err3}
&\sum_{n=1}^N\int_{t_{n-1}}^{t_n}\: \left\langle\ddot{\theta}(t),\dot{\theta}(t)\right\rangle \,\dd t +\sum_{n=1}^N\int_{t_{n-1}}^{t_n}\:\left(\dot{\theta}(t),\dot{\theta}(t)\right)_{L^2}\, \dd t+\sum_{n=1}^N\int_{t_{n-1}}^{t_n}\: \left(\nabla \theta(t), \nabla \dot{\theta}(t)\right)_{L^2} \dd t \nonumber &&\\
&\quad+\sum_{n=1}^{N-1}([\dot{\theta}]_{n}, \dot{\theta}(t_n^{+}))_{L^2}+\sum_{n=1}^{N-1}(\nabla [\theta]_{n}, \nabla \theta(t_{n}^{+}))_{L^2}+(\dot{\theta}(t_0^{+}), \dot{\theta}(t_0^{+}))_{L^2}+(\nabla\theta(t_0^{+}),\nabla \theta(t_0^{+}))_{L^2}\nonumber &&\\
&= -\sum_{n=1}^N\int_{t_{n-1}}^{t_n}\:\left(\dot{\rho}_1(t),\dot{\theta}(t)\right)_{L^2}\, \dd t-\sum_{n=1}^N\int_{t_{n-1}}^{t_n}\: \left(\nabla \rho_1(t), \nabla \dot{\theta}(t)\right)_{L^2} \dd t\nonumber &&\\
&\quad-\sum_{n=1}^{N-1}\left(\dot \rho_1(t_n^{-}), [\dot\theta]_n\right)_{L^2}-\sum_{n=1}^{N-1}\left([\nabla\rho_1]_n,\nabla\theta(t_n^{+})\right)_{L^2}-\left(\dot{\rho_1}(t_N^{-}), \dot{\theta}(t_N^{-})\right)_{L^2}-\left(\nabla\rho_1(t_0^{+}), \nabla\theta(t_0^{+})\right)_{L^2}\nonumber &&\\
&\quad -\sum_{n=1}^N\int_{t_{n-1}}^{t_n}\: \left\langle\ddot{\rho}_2(t),\dot{\theta}(t)\right\rangle \,\dd t-\sum_{n=1}^N\int_{t_{n-1}}^{t_n}\:\left(\dot{\rho}_2(t),\dot{\theta}(t)\right)_{L^2}\, \dd t. 
\end{flalign}
By (a)--(c) in Lemma \ref{projproperties}, we have $\rho_1(t_n^{\pm})=0$ and $\dot\rho_1(t_n^{-})=0$ for $n=1,\ldots,N$. Thus,
\begin{flalign}\label{theta_err4}
&\frac{1}{2}\|\dot{\theta}(t_0^{+})\|_{L^2}^2+\frac{1}{2}\sum_{n=1}^{N-1}\|[\dot{\theta}]_n\|_{L^2}^2 +\frac{1}{2}\|\dot{\theta}(t_N^{-})\|_{L^2}^2+\frac{1}{2}\|\nabla\theta(t_0^{+})\|_{L^2}^2+\frac{1}{2}\sum_{n=1}^{N-1}\|[\nabla\theta]_n\|_{L^2}^2 \nonumber&&\\
&\quad+\frac{1}{2}\|\nabla\theta(t_N^{-})\|_{L^2}^2  + \sum_{n=1}^N \int_{t_{n-1}}^{t_n}\|\dot \theta(t)\|_{L^2}^2\: \dd t\nonumber\\
&= -\sum_{n=1}^N\int_{t_{n-1}}^{t_n}\:\left(\dot{\rho}_1(t),\dot{\theta}(t)\right)_{L^2}\, \dd t-\sum_{n=1}^N\int_{t_{n-1}}^{t_n}\: \left(\nabla \rho_1(t), \nabla \dot{\theta}(t)\right)_{L^2} \dd t\nonumber &&\\
&\quad -\sum_{n=1}^N\int_{t_{n-1}}^{t_n}\: \left\langle\ddot{\rho}_2(t),\dot{\theta}(t)\right\rangle \,\dd t-\sum_{n=1}^N\int_{t_{n-1}}^{t_n}\:\left(\dot{\rho}_2(t),\dot{\theta}(t)\right)_{L^2}\, \dd t\nonumber &&\\
&\leq 2\sum_{n=1}^N \int_{t_{n-1}}^{t_n} \|\dot \rho_1(t)\|_{L^2}^2\: \dd t+2\sum_{n=1}^N \int_{t_{n-1}}^{t_n} \|\ddot \rho_2(t)\|_{L^2}^2\, \dd t +2\sum_{n=1}^N \int_{t_{n-1}}^{t_n} \|\dot \rho_2(t)\|_{L^2}^2\: \dd t \nonumber &&\\
&\quad+3\times\frac{1}{8}\sum_{n=1}^N \int_{t_{n-1}}^{t_n} \|\dot{\theta}(t)\|_{L^2}^2 \dd t +\sum_{n=1}^N \int_{t_{n-1}}^{t_n} \|\Delta \rho_1(t)\|_{L^2} \|\dot{\theta}(t)\|_{L^2}\ \,\dd t \nonumber &&\\
&\leq 2\sum_{n=1}^N \int_{t_{n-1}}^{t_n} \|\dot \rho_1(t)\|_{L^2}^2\: \dd t +2\sum_{n=1}^N \int_{t_{n-1}}^{t_n} \|\dot \rho_2(t)\|_{L^2}^2\: \dd t +2\sum_{n=1}^N \int_{t_{n-1}}^{t_n} \|\ddot \rho_2(t)\|_{L^2}^2\, \dd t \nonumber &&\\
& \quad+2\sum_{n=1}^N \int_{t_{n-1}}^{t_n} \|\Delta \rho_1(t)\|_{L^2}^2 \,\dd t +\frac{1}{2}\sum_{n=1}^N \int_{t_{n-1}}^{t_n}\|\dot{\theta}(t)\|_{L^2}^2 \,\dd t&&\nonumber\\
&\leq  2C\sum_{n=1}^N \frac{k_n^{2q_n}}{q_n^{2(s_n-1)}}\|u\|_{H^{s_n}(I_n; L^2(\Omega))}^2+2C\sum_{n=1}^N  h^{2r+2}\|u\|_{H^1(I_n;H^{r+1}(\Omega))}^2+2C\sum_{n=1}^N  h^{2r+2}\|u\|_{H^2(I_n;H^{r+1}(\Omega))}^2 \nonumber &&\\
&\quad+2C\sum_{n=1}^N \frac{k_n^{2q_n+2}}{q_n^{2(s_n-1)}} \|u\|_{H^{s_n}(I_n; H^2(\Omega))}^2+\frac{1}{2}\sum_{n=1}^N \int_{t_{n-1}}^{t_n}\|\dot{\theta}(t)\|_{L^2}^2\ \,\dd t\nonumber&&\\
&\leq C \sum_{n=1}^N \left(\frac{k_n^{2q_n}}{q_n^{2(s_n-1)}} +h^{2r+2}\right)\|u\|_{H^{s_n}(I_n; H^{r+1}(\Omega))}^2 +\frac{1}{2}\sum_{n=1}^N \int_{t_{n-1}}^{t_n}\|\dot{\theta}(t)\|_{L^2}^2\ \,\dd t.&&
\end{flalign}
Since $$\frac{1}{2}\sum_{n=1}^N \int_{t_{n-1}}^{t_n}\|\dot{\theta}(t)\|_{L^2}^2 \,\dd t\leq \frac{1}{2}|||\theta|||^2,$$ we can absorb the last term on the right-hand side of (\ref{theta_err4}) into the left-hand side to obtain that 
\begin{equation}\label{theta_error}
|||\theta|||^2  \leq 2C  \sum_{n=1}^N \frac{k_n^{2q_n}}{q_n^{2s_n-2}}\|u\|_{H^{s_n}(I_n; H^{r+1}(\Omega))}^2+2Ch^{2r+2}\|u\|_{H^{s_n}(0,T; H^{r+1}(\Omega))}^2. 
\end{equation}
Note that 
\begin{flalign*}
|||\rho_2|||^2&=  \frac{1}{2}\|\dot{\rho_2}(t_0^{+})\|_{L^2}^2+\frac{1}{2}\sum_{n=1}^{N-1}\|[\dot{\rho}_2]_n\|_{L^2}^2 +\frac{1}{2}\|\dot{\rho}_2(t_N^{-})\|_{L^2}^2+\frac{1}{2}\|\nabla\rho_2(t_0^{+})\|_{L^2}^2+\frac{1}{2}\sum_{n=1}^{N-1}\|[\nabla\rho_2]_n\|_{L^2}^2 \nonumber &&\\
&\quad+\frac{1}{2}\|\nabla\rho_2(t_N^{-})\|_{L^2}^2  + \sum_{n=1}^N \int_{t_{n-1}}^{t_n}\|\dot \rho_2(t)\|_{L^2}^2\: \dd t\nonumber &&\\
&=  \frac{1}{2}\|\dot{\rho_2}(t_0^{+})\|_{L^2}^2+\frac{1}{2}\|\dot{\rho}_2(t_N^{-})\|_{L^2}^2+\frac{1}{2}\|\nabla\rho_2(t_0^{+})\|_{L^2}^2+\frac{1}{2}\|\nabla\rho_2(t_N^{-})\|_{L^2}^2  + \sum_{n=1}^N \int_{t_{n-1}}^{t_n}\|\dot \rho_2(t)\|_{L^2}^2\: \dd t\nonumber &&\\
&\leq C h^{2r+2} \left(\|\dot{u}(t_0^{+})\|_{H^{r+1}}^2+\|\dot{u}(t_N^{-})\|_{H^{r+1}}^2\right) +C h^{2r} \left( \| u(t_0^{+})\|_{H^{r+1}}^2 +\| u(t_N^{-})\|_{H^{r+1}}^2\right)&&\\
&\quad+Ch^{2r+2}\|u\|_{H^1(0,T; H^{r+1}(\Omega))}^2. &&\\
\end{flalign*}
We also have that 
\begin{flalign*}
|||\rho_1|||^2&=  \frac{1}{2}\|\dot{\rho_1}(t_0^{+})\|_{L^2}^2+\frac{1}{2}\sum_{n=1}^{N-1}\|[\dot{\rho}_1]_n\|_{L^2}^2 +\frac{1}{2}\|\dot{\rho}_1(t_N^{-})\|_{L^2}^2+\frac{1}{2}\|\nabla\rho_1(t_0^{+})\|_{L^2}^2+\frac{1}{2}\sum_{n=1}^{N-1}\|[\nabla\rho_1]_n\|_{L^2}^2 \nonumber &&\\
&\quad+\frac{1}{2}\|\nabla\rho_1(t_N^{-})\|_{L^2}^2  + \sum_{n=1}^N \int_{t_{n-1}}^{t_n}\|\dot \rho_1(t)\|_{L^2}^2\: \dd t\nonumber\\
&=  \frac{1}{2}\sum_{n=1}^N \|\dot\rho_1(t_{n-1}^{+})\|_{L^2}^2+ \sum_{n=1}^N \int_{t_{n-1}}^{t_n}\|\dot \rho_1(t)\|_{L^2}^2\: \dd t\nonumber &&\\
&=  \frac{1}{2}\sum_{n=1}^N \|-\int_{t_{n-1}}^{t_n}\ddot\rho_1(s)\, \dd s\|_{L^2}^2+ \sum_{n=1}^N \int_{t_{n-1}}^{t_n}\|\dot \rho_1(t)\|_{L^2}^2\: \dd t\nonumber &&\\
&\leq \sum_{n=1}^N\frac{k_n}{2}\|\partial_{tt}(\Pi_{I_n}^{q_n}\mathcal{P}_h u-\mathcal{P}_hu)\|_{L^2(I_n; L^2(\Omega))}^2+\sum_{n=1}^N\|\partial_{t}(\Pi_{I_n}^{q_n}\mathcal{P}_h u-\mathcal{P}_hu)\|_{L^2(I_n; L^2(\Omega))}^2\\
&\leq  C\sum_{n=1}^N \frac{k_n^{2q_n-1}}{q_n^{2(s_n-3)}}\|\mathcal{P}_h u\|_{H^{s_n}(I_n; L^2(\Omega))}^2 +C\sum_{n=1}^N \frac{k_n^{2q_n}}{q_n^{2(s_n-1)}}\|\mathcal{P}_h u\|_{H^{s_n}(I_n; L^2(\Omega))}^2 \nonumber &&\\
&\leq C\sum_{n=1}^N \frac{k_n^{2q_n-1}}{q_n^{2s_n-6}}\|u\|_{H^{s_n}(I_n; L^2(\Omega))}^2.&&
\end{flalign*}
By the triangle's inequality, we deduce that 
\begin{align}\label{error_in_energy2}
|||u-u_{\mathrm{DG}}||| & \leq |||\theta|||+|||\rho_1|||+|||\rho_2||| \nonumber\\
&\leq C\bigg\{\sum_{n=1}^N \frac{k_n^{2q_n-1}}{q_n^{2s_n-6}} \|u\|_{H^{s_n}(I_n; H^{r+1}(\Omega))}^2+ h^{2r+2}\|u\|_{H^{s_n}(0,T; H^{r+1}(\Omega))}^2\bigg\}^{\frac{1}{2}} \nonumber\\
&\quad+ \bigg\{h^{2r+2} \left(\|\dot{u}(t_0^{+})\|_{H^{r+1}}^2+\|\dot{u}(t_N^{-})\|_{H^{r+1}}^2\right) + h^{2r} \left( \| u(t_0^{+})\|_{H^{r+1}}^2 +\| u(t_N^{-})\|_{H^{r+1}}^2\right)\bigg\}^{\frac{1}{2}}.
\end{align}
\end{proof}

\begin{remark}\label{uniform_grid2}
If we use uniform time intervals $k_n= k$, and uniform polynomial degrees $q_n=q\geq 2, s_n =s\geq 2$ for $n=1,\ldots, N$, then the error bound becomes
\begin{align*}
|||u-u_{\mathrm{DG}}||| &\leq C\bigg\{ h^{2r+2} \left(\|\dot{u}(t_0^{+})\|_{H^{r+1}}^2+\|\dot{u}(t_N^{-})\|_{H^{r+1}}^2\right) + h^{2r} \left( \|u(t_0^{+})\|_{H^{r+1}}^2 +\|u(t_N^{-})\|_{H^{r+1}}^2\right) \bigg\}^{\frac{1}{2}} \nonumber\\
&\quad +\bigg\{\frac{k^{2q-1}}{q^{2s-6}} \|u\|_{H^{s}(0,T; H^{r+1}(\Omega))}^2+ h^{2r+2}\|u\|_{H^{s}(0,T; H^{r+1}(\Omega))}^2 \bigg\}^{\frac{1}{2}}\nonumber\\
&\leq C(u)\left( h^{2r}+ \frac{k^{2q-1}}{q^{2s-6}} \right)^{\frac{1}{2}}, 
\end{align*}
where $C(u)$ is a constant depending on the exact solution $u$.
\end{remark}
\begin{remark}\label{l2_end}
If we only consider the $L^2$ error estimate at the end point, we have 
\begin{flalign*}
\|\dot{u}(t_N^{-})-\dot u_{\mathrm{DG}}(t_N^{-})\|_{L^2}&\leq  \|\dot\theta(t_N^{-})\|_{L^2}+\|\dot\rho_1(t_N^{-})\|_{L^2}+\|\dot\rho_2(t_N^{-})\|_{L^2}&&\\
&=  \|\dot\theta(t_N^{-})\|_{L^2}+\|\dot\rho_2(t_N^{-})\|_{L^2}&&\\
&\leq  C\left(h^{2(r+1)}+ \frac{k^{2q}}{q^{2s-2}}\right)^{\frac{1}{2}}\|u\|_{H^{s}(0,T; H^{r+1}(\Omega))}+Ch^{r+1}\|\dot{u}(t_N^{-})\|_{L^2} &&\\
& \leq C(u) \left(h^{2(r+1)}+ \frac{k^{2q}}{q^{2s-2}}\right)^{\frac{1}{2}}.&&
\end{flalign*}
Analogously,  
\begin{flalign*}
\|\dot{u}(t_N^{-})-\dot u_{\mathrm{DG}}(t_N^{-})\|_{L^2}+\| u(t_N^{-})- u_{\mathrm{DG}}(t_N^{-})\|_{L^2}\leq  &\: C(u)\left(h^{2(r+1)}+ \frac{k^{2q}}{q^{2s-2}}\right)^{\frac{1}{2}}. &&
\end{flalign*}
\end{remark}
\section{Numerical experiments}\label{numerical}
In this section, we show some numerical results to verify the convergence properties of the DG-in-time scheme introduced in  Section \ref{fullydiscrete}.
\subsection{Numerical results for a scalar linear wave equation}\label{linear_wave_experiments}
We consider the one-dimensional wave problem
$$\begin{cases}
\ddot u(x,t)+2\gamma\dot u(x,t)+\gamma^2u(x,t)-\partial_{xx} u(x,t))=f(x,t) &\quad \text{ in } (0,1)\times (0,T],\\
u(0, t) = u(1,t)= 0 &\quad\text{for all } t\in (0,T], \\
u(x,0)=u_0(x), \quad \dot u(x,0)=u_1(x).
\end{cases}$$
Here, we set $T=1$, $\gamma=1$ and let $u_0$, $u_1$ and $f$ be chosen such that the exact solution is 
$$u(x,t)=\sin(\sqrt{2}\pi t)\sin(\pi x).$$
That is, $u_0(x)\equiv 0$, $u_1(x)= \sqrt{2}\pi \sin(\pi x)$, and $$f(x,t)=[(-\pi^2+\gamma^2)\sin(\sqrt{2}\pi t)+2\sqrt{2}\gamma\pi\cos(\sqrt{2}\pi t)]\sin(\pi x).$$
We first discretize the problem in the spatial direction using continuous piecewise polynominals of degree $r\geq 1$, and compute the corresponding mass and stiffness matrices using FEniCS. Let $\mathcal{V}_h$ be the finite element function space with $h$ being the spatial discretization parameter specified previously. The numerical approximation of the one-dimensional wave-type equation is the following: find $u_h\in \mathcal{V}_h$ such that 
$$\int_{\Omega} \ddot{u}_h\cdot v_h\: \mathrm{d} x + \int_{\Omega} 2\gamma\dot{u}_h\cdot v_h\: \mathrm{d} x+\int_{\Omega} \gamma^2 u_h\cdot v_h\: \mathrm{d} x+ \int_{\Omega}\partial_x u_h \partial_x v_h\: \mathrm{d} x = \int_{\Omega} f\cdot v_h\: \mathrm{d} x,$$
for all $v_h\in\mathcal{V}_h$.
After discretization in space based on the above weak formulation, the problem results in the following second-order differential system for the nodal displacement $\mathbf{U}(t)$:
$$\begin{cases} \tilde{M} \ddot{\mathbf{U}}(t)+2\gamma \tilde{M}\dot{\mathbf{U}}(t)+\gamma^2\tilde{M}\mathbf{U}(t)+\tilde{K}\mathbf{U}(t)= \mathbf{F}(t), \quad t\in (0,T], &\\
\dot{\mathbf{U}}(0)=\mathbf{U}_1, \\
\mathbf{U}(0)=\mathbf{U}_0,
\end{cases}$$
where $\ddot{\mathbf{U}}(t)$ (respectively $\dot{\mathbf{U}}(t)$) represents the vector of nodal acceleration (respectively velocity) and $\mathbf{F}(t)$ is the vector of externally applied loads. Here, $\tilde{M}$ and $\tilde{K}$ are the mass and stiffness matrices (with a Dirichlet boundary condition applied) whose entries are, respectively, 
$$\tilde{M}_{ij}:=\int_0^1 \psi_i(x)\psi_j(x) \: \mathrm{d} x,$$
$$\tilde{K}_{ij}:=\int_0^1 \partial_x\psi_i(x)\partial_x\psi_j(x) \: \mathrm{d} x,$$
where $\{\psi_i\}_{i=1}^r$ are the basis functions in the spatial direction. 

Multiplying the above algebraic formulation by $\tilde{M}^{-\frac{1}{2}}$ and setting $\mathbf{Z}(t)=\tilde{M}^{\frac{1}{2}}\mathbf{U}(t)$, we obtain
\begin{equation}\label{algebra_eq}
\mathbf{\ddot{Z}}(t) + L\mathbf{\dot{Z}}(t)+ K\mathbf{Z}(t) =\mathbf{G}(t), \quad t\in (0,T],
\end{equation}
\begin{equation}\label{algebra_ic}
\mathbf{\dot{Z}}(0)= \tilde{M}^{\frac{1}{2}}\mathbf{U}_1, \quad \mathbf{Z}(0)=\tilde{M}^{\frac{1}{2}} \mathbf{U}_0,
\end{equation}
where $\mathbf{U}_0=[0, \ldots,0]^{\mathrm{T}}\in \mathbb{R}^{\hat{d}}$ and $\mathbf{U}_1$ is the $\hat{d}$-vector corresponding to $u_1$ at the grid points. Here $L=2\gamma\mathrm{Id}$, $K=\gamma^2\mathrm{Id}+\tilde{M}^{-\frac{1}{2}}\tilde{K}\tilde{M}^{-\frac{1}{2}}$ and $\mathbf{G}(t)=\tilde{M}^{-\frac{1}{2}}\mathbf{F}(t).$

We subdivided $[0,T)$ with $T=1$ into $N$ subintervals $I_n$, for $n=1,\ldots,N,$ of uniform length $k$. We assume that the polynomial degree in time is constant at each time step. That is, $q_1=\cdots=q_N\geq 2.$ If we consider the time integration on a generic time interval $I_n$ for each $n=1,\ldots N$, our discontinuous-in-time formulation reads as follows:
find $\mathbf{Z}\in\mathcal{V}_{kh}^{q_n}$ such that 
\begin{flalign}\label{weak_form_t}
&\left(\mathbf{\ddot{Z}}(t), \dot{\mathbf{v}}\right)_{L^2(I_n)}+\left(L\mathbf{\dot{Z}}(t),\dot{\mathbf{v}}\right)_{L^2(I_n)}+\left(K\mathbf{Z}(t),\dot{\mathbf{v}}\right)_{L^2(I_n)}+\mathbf{\dot{Z}}(t_{n-1}^{+})\cdot \dot{\mathbf{v}}(t_{n-1}^{+})+K\mathbf{Z}(t_{n-1}^{+})\cdot\mathbf{v}(t_{n-1}^{+}) &&\nonumber\\
&= \left( \mathbf{G}(t), \dot{\mathbf{v}}\right)_{L^2(I_n)}+\mathbf{\dot{Z}}(t_{n-1}^{-})\cdot\dot{\mathbf{v}}(t_{n-1}^{+})+K\mathbf{Z}(t_{n-1}^{-})\cdot\mathbf{v}(t_{n-1}^{+}), \nonumber &&
\end{flalign}
for all $\mathbf{v}\in \mathcal{V}_{kh}^{q_n}$, where on the right-hand side the values $\mathbf{\dot{Z}}(t_{n-1}^{-})$ and $\mathbf{Z}(t_{n-1}^{-})$ computed for $I_{n-1}$ are used as initial conditions for the current time interval. For $I_1$, we set $\mathbf{\dot{Z}}(t_0^{-})=\mathbf{\dot{Z}}(0)$ and $\mathbf{Z}(t_0^{-})=\mathbf{Z}(0).$ Focusing on the generic time interval $I_n$, we introduce the basis functions in the time direction $\{\phi^{j}(t)\}_{j=1}^{q_n+1}$ for the polynomial space $\mathbb{P}^{q_n}(I_n)$ and define $D=\hat{d}(q_n+1),$ the dimension of the local finite element space $\mathcal{V}_{kh}^{q_n}.$ We also introduce the vectorial basis $\{\Phi_{m}^j(t)\}_{m=1,\ldots,\hat{d}}^{j=1,\ldots, q_n+1}$, where $\Phi_{m}^j(t)$ is the $\hat{d}$-dimensional vector whose $m$-th component is $\phi^j(t)$ and the other components are zero. We write 
\begin{equation}\label{basis_decom_in_t}
\mathbf{Z}(t)=\sum_{m=1}^{\hat{d}} \sum_{j=1}^{q_n+1} \alpha_{m}^j \Phi_m^j(t),
\end{equation}
where $\alpha_m^j\in\mathbb{R}$ for $m=1,\ldots, \hat{d}$, $j=1,\ldots,q_n+1$.
By choosing $\mathbf{v}(t)= \Phi_{m}^j(t)$ for each $m=1,\ldots,\hat{d}$, $j=1,\ldots, q_n+1$, we obtain the following algebraic system
\begin{equation}\label{final_algebraic}
\mathbf{A}\mathbf{z}=\mathbf{b},
\end{equation}
where $\mathbf{z}\in\mathbb{R}^{D}=\mathbb{R}^{(q_n+1)\hat{d}}$ is the solution vector (values for $\alpha_m^j$). Here $\mathbf{b}\in\mathbb{R}^D$ corresponds to the right-hand side, which is given componentwise as 
\begin{equation}\label{rhsb}
\mathbf{b}_m^j =  \left( \mathbf{G}(t), \dot{\Phi}_m^j\right)_{L^2(I_n)}+\mathbf{\dot{Z}}(t_{n-1}^{-})\cdot\dot{\Phi}_m^j(t_{n-1}^{+})+K\mathbf{Z}(t_{n-1}^{-})\cdot\Phi_m^j(t_{n-1}^{+}), 
\end{equation}
for $m=1,\ldots,\hat{d}$, $j=1,\ldots,q_n+1.$
$\mathbf{A}$ is the local stiffness matrix with its structure being discussed below.  For $l,j=1,\ldots,q_n+1$, 
\begin{align*}
&M_{lj}^1=\:\left(\ddot{\phi}^j, \dot{\phi}^l\right)_{L^2(I_n)},\quad M_{lj}^2=\:\left(\dot{\phi}^j, \dot{\phi}^l\right)_{L^2(I_n)},\quad M_{lj}^3=\:\left(\phi^j, \dot{\phi}^l\right)_{L^2(I_n)}, \\
& M_{lj}^4=\:\dot{\phi}^j(t_{n-1}^{+})\cdot \dot{\phi}^l(t_{n-1}^{+}), \quad M_{lj}^5=\:\phi^j(t_{n-1}^{+})\cdot \phi^l(t_{n-1}^{+}).
\end{align*}
Setting 
\begin{align*}
M= & \:M^1+M^4,\\
B_{ij}= &\: L_{ij} M^2+K_{ij}(M^3+M^5),
\end{align*}
with $M, B_{ij}\in\mathbb{R}^{(q_n+1)\times (q_n+1)}$ for any $i,j=1,\ldots,\hat{d}$, we can rewrite the matrix $\mathbf{A}$ as 
$$\mathbf{A}= \begin{bmatrix}
M & 0 & 0  &\cdots & 0  \\
0 & M & 0 &\cdots & 0 \\
\vdots& \ddots  & \ddots &\ddots& \vdots & \\
0  & 0 & 0 &\cdots & M
\end{bmatrix}+\begin{bmatrix}
B_{1,1} & B_{1,2} & \cdots & B_{1,d}  \\
B_{2,1} & B_{2,2}& \cdots & B_{2,d} \\
\vdots& \ddots & \ddots  & \vdots & \\
B_{d,1} & B_{d,2}  &\cdots & B_{d,d}
\end{bmatrix}.$$
For each time interval $I_n=(t_{n-1}, t_n]$, we use the following shifted Legendre polynomials $\{\phi_i\}$ as the basis polynomials:
\begin{align*}
& \phi^1(t)=1, \quad \phi^2(t)=\frac{2(t-t_{n-1}^{+})}{k_n}-1, \quad \phi^3(t)=\frac{6(t-t_{n-1}^{+})^2}{k_n^2}-\frac{6(t-t_{n-1}^{+})}{k_n}+1, \\
{}& \phi^4(t)=\frac{20(t-t_{n-1}^{+})^3}{k_n^3}-\frac{30(t-t_{n-1}^{+})^2}{k_n^2}+\frac{12(t-t_{n-1}^{+})}{k_n}-1,\\
{} & \phi^5(t)=\frac{70(t-t_{n-1}^{+})^4}{k_n^4}-\frac{140(t-t_{n-1}^{+})^3}{k_n^3}+\frac{90(t-t_{n-1}^{+})^2}{k_n^2}-\frac{20(t-t_{n-1}^{+})}{k_n}+1,\\
{} & \phi^6(t)=\frac{252(t-t_{n-1}^{+})^5}{k_n^5}-\frac{630(t-t_{n-1}^{+})^4}{k_n^4}+\frac{560(t-t_{n-1}^{+})^3}{k_n^3}-\frac{210(t-t_{n-1}^{+})^2}{k_n^2}+\frac{30(t-t_{n-1}^{+})}{k_n}-1.
\end{align*}
This implies that 
\begin{align*}
&\dot{\phi}^1(t)=0, \quad \dot{\phi}^2(t)=\frac{2}{k_n}, \quad \dot{\phi}^3(t)=\frac{12(t-t_{n-1}^{+})}{k_n^2}-\frac{6}{k_n},\\ 
&\dot{\phi}^4(t)=\frac{60(t-t_{n-1}^{+})^2}{k_n^3}-\frac{60(t-t_{n-1}^{+})}{k_n^2}+\frac{12}{k_n},\\
&\dot{\phi}^5(t)=\frac{280(t-t_{n-1}^{+})^3}{k_n^4}-\frac{420(t-t_{n-1}^{+})^2}{k_n^3}+\frac{180(t-t_{n-1}^{+})}{k_n^2}-\frac{20}{k_n},\\
&\dot{\phi}^6(t)=\frac{1260(t-t_{n-1}^{+})^4}{k_n^5}-\frac{2520(t-t_{n-1}^{+})^3}{k_n^4}+\frac{1680(t-t_{n-1}^{+})^2}{k_n^3}-\frac{420(t-t_{n-1}^{+})}{k_n^2}+\frac{30}{k_n},
\end{align*}
and
\begin{align*}
&\ddot{\phi}^1(t)=0, \quad\ddot{\phi}^2(t)=0, \quad \ddot{\phi}^3(t)=\frac{12}{k_n^2},\quad\ddot{\phi}^4(t)=\frac{120(t-t_{n-1}^{+})}{k_n^3}-\frac{60}{k_n^2}, \\
&\ddot{\phi}^5(t)=\frac{840(t-t_{n-1}^{+})^2}{k_n^4}-\frac{840(t-t_{n-1}^{+})}{k_n^3}+\frac{180}{k_n^2},\\
&\ddot{\phi}^6(t)=\frac{5040(t-t_{n-1}^{+})^3}{k_n^5}-\frac{7560(t-t_{n-1}^{+})^2}{k_n^4}+\frac{3360(t-t_{n-1}^{+})}{k_n^3}-\frac{420}{k_n^2}.
\end{align*}
We compute $M_1, \cdots M_5$ according to the above-mentioned formulae for $q_n=2,3,4,5$ respectively.

Here we use $CG$--$r$ elements (continuous piecewise polynomials of degree $r$) in space with $h=k$, $T=1$ and $\gamma=1$, and compute the errors $|||\mathbf{Z}-\mathbf{Z}_{\mathrm{DG}}|||$ and $\|\dot{u}(T)-\dot{u}_{\mathrm{DG}}(t_N^{-})\|_{L^2}$ versus $k$ for $k=2^{-l}$, $l=1, 2,3,4,5$, with respect to polynomial degrees $q=2,3,4,5$. We choose $r=q-1$, i.e. the polynomial degree in the spatial direction is one order less than the polynomial degree in the time direction. Both the semi-discrete errors in the energy norm and the fully discrete errors in the $L^2$-norm are computed in Table \ref{tab:table1} and plotted in a log-log scale in 
Fig. \ref{fig:plot1} and Fig. \ref{fig:plot2}.

\begin{table}[h!]
  \begin{center}
    \caption{Computed errors $|||\mathbf{Z}-\mathbf{Z}_{\mathrm{DG}}|||$,$||\dot{u}(T)-\dot{u}_{\mathrm{DG}}(t_N^{-})||_{L^2} $ and corresponding convergence rates with respect to polynomial degrees $q=2, 3, 4,5$.}
    \label{tab:table1}
    \begin{tabular}{c c c c c c } 
     \hline
      $ q$  &    $k$   &  energy-norm error  & rate  & $L^2$-error & rate \\
      \hline 
  $2$  & $5.000\mathrm{e}-1$    & $ 1.6504\mathrm{e}-0$   & ---  & $5.6323\mathrm{e}-1$ & ---  \\
       & $2.500\mathrm{e}-1$    &  $6.5087\mathrm{e}-1$  &   $1.3424$                                                                   &$1.5238\mathrm{e}-1$ &  $1.8861$\\
      & $1.250\mathrm{e}-1$    &  $ 2.3300\mathrm{e}-1$   &   $1.4820$&$3.8942\mathrm{e}-2$ &  $1.9683$\\
&    $6.250\mathrm{e}-2$    &  $8.3340\mathrm{e}-2$      &  $1.4832$ & $9.7781\mathrm{e}-3$ &  $1.9937$\\
&    $3.125\mathrm{e}-2$    &  $2.9431\mathrm{e}-2$      &  $1.5017$ & $2.4452\mathrm{e}-3$ &  $1.9996$\\
 
$3$  & $5.000\mathrm{e}-1$    & $  3.3051\mathrm{e}-1$   & ---  & $2.1979\mathrm{e}-2$ & ---  \\
       & $2.500\mathrm{e}-1$    &  $6.6921\mathrm{e}-2$  &   $2.3041$                                                                   &$2.5286\mathrm{e}-3$ &  $3.1197$\\
      & $1.250\mathrm{e}-1$    &  $ 1.2170\mathrm{e}-2$   &   $2.4591$&$2.9962\mathrm{e}-4$ &  $3.0771$\\
&    $6.250\mathrm{e}-2$    &  $2.1682\mathrm{e}-3$      &  $2.4888$ & $3.6708\mathrm{e}-5$ &  $3.0290 $\\
&    $3.125\mathrm{e}-2$    &  $3.8421\mathrm{e}-4$      &  $2.4965$ & $4.5613\mathrm{e}-6$ &  $3.0086$\\
  $4$  & $5.000\mathrm{e}-1$    & $  6.3950\mathrm{e}-2$   & ---  & $1.9566\mathrm{e}-3$ & ---  \\
       & $2.500\mathrm{e}-1$    &  $5.7749\mathrm{e}-3$  &   $3.4691$                                                                   &$1.2436\mathrm{e}-4$ &  $3.9758$\\
      & $1.250\mathrm{e}-1$    &  $ 5.1721\mathrm{e}-4$   &   $3.4810$&$7.7114\mathrm{e}-6$ &  $4.0114$\\
&    $6.250\mathrm{e}-2$    &  $4.6070\mathrm{e}-5$      &  $3.4889$ & $4.8969\mathrm{e}-7$ &  $3.9771 $\\
&    $3.125\mathrm{e}-2$    &  $4.1121\mathrm{e}-6$      &  $3.4859$ & $3.0656\mathrm{e}-8$ &  $3.9976$\\
  $5$  & $5.000\mathrm{e}-1$    & $  6.8763\mathrm{e}-3$   & ---  & $1.5180\mathrm{e}-4$ & ---  \\
       & $2.500\mathrm{e}-1$    &  $3.3619\mathrm{e}-4$  &   $4.3542$                                                                   &$4.3686\mathrm{e}-6$ &  $5.1189$\\
      & $1.250\mathrm{e}-1$    &  $ 1.5264\mathrm{e}-5$   &   $4.4669$&$1.2188\mathrm{e}-7$ &  $5.1636$\\
&    $6.250\mathrm{e}-2$    &  $ 6.9025\mathrm{e}-7$      &  $4.4669$ & $3.8640\mathrm{e}-9$ &  $4.9792 $\\
&    $3.125\mathrm{e}-2$    &  $3.1659\mathrm{e}-8$      &  $4.4464$ & $1.4264\mathrm{e}-10$ &  $4.7586$\\
\hline        
    \end{tabular}
  \end{center}
\end{table}
As expected, the convergence rate in the energy norm is of order $O(k^{q-\frac{1}{2}})$, which is consistent with Remark \ref{uniform_grid}. The $L^2$-error decreases as the time step $k$ deceases. In particular, the convergence rate of $O(k^{q})$ is observed. This agrees with our theoretical result when $h=k$ and $r=q-1$ (cf. Remark \ref{l2_end}). If we use $CG$-$1$ elements in space instead, the convergence rate will remain the same as we increase the polynomial degree from $q=2$ to $q=3,4,5$. This is because  now $O(h^2)$ is dominating $O(k^{q})$ for $q>2$ in the error estimate (cf. Table \ref{tab:table2}).
\begin{figure}[h!]
 \centering
  \begin{minipage}[b]{0.7\textwidth}
    \includegraphics[width=\textwidth]{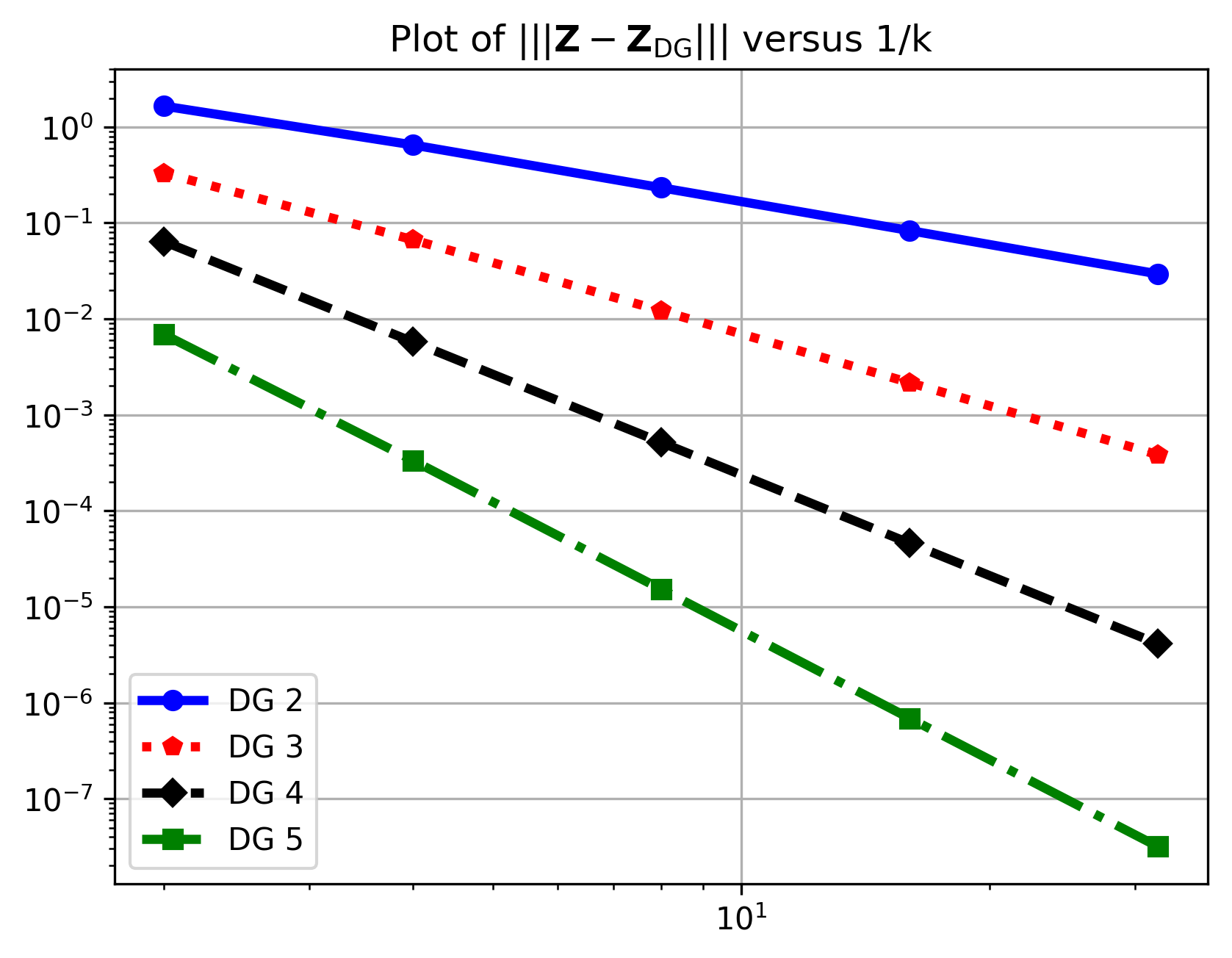}
\caption{$|||\mathbf{Z}-\mathbf{Z}_{\mathrm{DG}}|||$ versus $1/k$, with $k=2^{-l}$ for $l=1,2,3,4,5$ in a log-log scale.}\label{fig:plot1}
  \end{minipage}
  \hfill
  \begin{minipage}[b]{0.7\textwidth}
    \includegraphics[width=\textwidth]{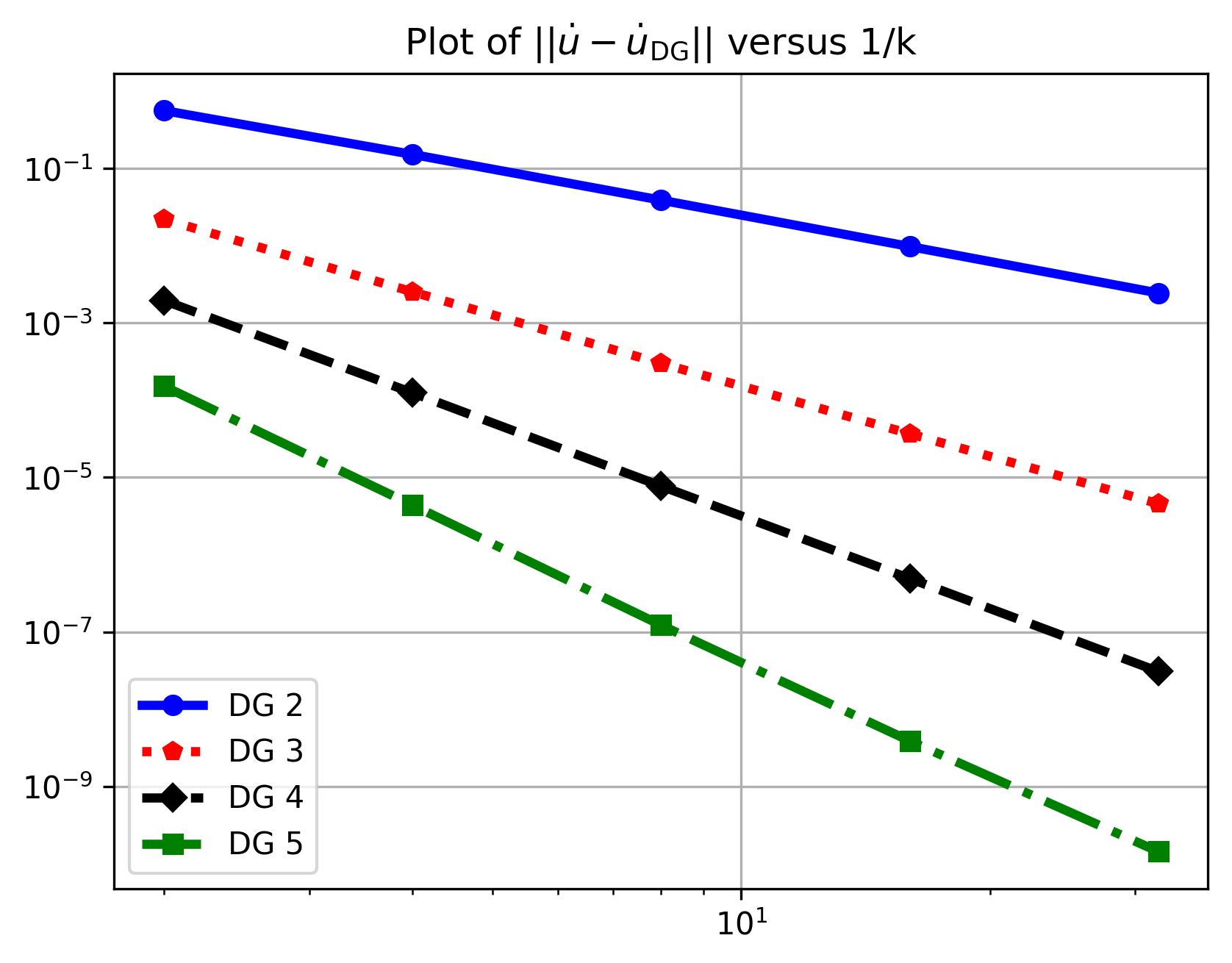}
\caption{$\|\dot{u}-\dot{u}_{\mathrm{DG}}\|_{L^2}$ versus $1/k$, with $k=2^{-l}$ for $l=1,2,3,4,5$  in a log-log scale.}\label{fig:plot2}
  \end{minipage}
\end{figure}
\begin{table}[h!]
  \begin{center}
  \caption{Computed errors $|||\mathbf{Z}-\mathbf{Z}_{\mathrm{DG}}|||$, $||\dot{u}(T)-\dot{u}_{\mathrm{DG}}(t_N^{-})||_{L^2} $ for $CG$--$1$ elements in space and corresponding convergence rates with respect to polynomial degrees $q=2, 3, 4, 5$.}
    \label{tab:table2}
    \begin{tabular}{c c c c c c } 
     \hline
      $ q$  &    $k$   &  energy-norm error  & rate  & $L^2$-error & rate \\
      \hline 
  $2$  & $5.000\mathrm{e}-1$    & $ 1.6504\mathrm{e}-0$   & ---  & $5.6323\mathrm{e}-1$ & ---  \\
       & $2.500\mathrm{e}-1$    &  $6.5087\mathrm{e}-1$  &   $1.3424$                                                                   &$1.5238\mathrm{e}-1$ &  $1.8861$\\
      & $1.250\mathrm{e}-1$    &  $ 2.3300\mathrm{e}-1$   &   $1.4820$&$3.8942\mathrm{e}-2$ &  $1.9683$\\
&    $6.250\mathrm{e}-2$    &  $8.3340\mathrm{e}-2$      &  $1.4832$ & $9.7781\mathrm{e}-3$ &  $1.9937$\\
&    $3.125\mathrm{e}-2$    &  $2.9431\mathrm{e}-2$      &  $1.5017$ & $2.4452\mathrm{e}-3$ &  $1.9996$\\
 
  $3$  & $5.000\mathrm{e}-1$    & $  7.0451\mathrm{e}-1$   & ---  & $5.9765\mathrm{e}-1$ & ---  \\
       & $2.500\mathrm{e}-1$    &  $1.9493\mathrm{e}-1$  &   $1.8537$                                                                   &$1.5648\mathrm{e}-1$ &  $1.9333$\\
      & $1.250\mathrm{e}-1$    &  $ 4.8544\mathrm{e}-2$   &   $2.0056$&$3.9116\mathrm{e}-2$ &  $2.0001$\\
&    $6.250\mathrm{e}-2$    &  $1.2016\mathrm{e}-2$      &  $2.0143$ & $9.7726\mathrm{e}-3$ &  $2.0009 $\\
&    $3.125\mathrm{e}-2$    &  $2.9838\mathrm{e}-3$      &  $2.0097$ & $2.4426\mathrm{e}-3$ &  $2.0003$\\
  $4$  & $5.000\mathrm{e}-1$    & $  6.5331\mathrm{e}-1$   & ---  & $6.1114\mathrm{e}-1$ & ---  \\
       & $2.500\mathrm{e}-1$    &  $1.8357\mathrm{e}-1$  &   $1.8314$                                                                   &$1.5677\mathrm{e}-1$ &  $1.9629$\\
      & $1.250\mathrm{e}-1$    &  $ 4.7005\mathrm{e}-2$   &   $1.9654$&$3.9124\mathrm{e}-2$ &  $2.0025$\\
&    $6.250\mathrm{e}-2$    &  $1.1819\mathrm{e}-2$      &  $1.9917$ & $9.7728\mathrm{e}-3$ &  $2.0012 $\\
&    $3.125\mathrm{e}-2$    &  $2.9590\mathrm{e}-3$      &  $1.9979$ & $2.4426\mathrm{e}-3$ &  $2.0004$\\
   $5$  & $5.000\mathrm{e}-1$    & $  6.5001\mathrm{e}-1$   & ---  & $6.1136\mathrm{e}-1$ & ---  \\
       & $2.500\mathrm{e}-1$    &  $1.8348\mathrm{e}-1$  &   $1.8248$                                                                   &$1.5677\mathrm{e}-1$ &  $1.9934$\\
      & $1.250\mathrm{e}-1$    &  $ 4.7002\mathrm{e}-2$   &   $1.9648$&$3.9124\mathrm{e}-2$ &  $2.0025$\\
&    $6.250\mathrm{e}-2$    &  $1.1819\mathrm{e}-2$      &  $1.9916$ & $9.7728\mathrm{e}-3$ &  $2.0012 $\\
&    $3.125\mathrm{e}-2$    &  $2.9590\mathrm{e}-3$      &  $1.9979$ & $2.4426\mathrm{e}-3$ &  $2.0004$\\
\hline        
       
    \end{tabular}
  \end{center}
\end{table}
\subsection{Numerical results for a two-dimensional elastodynamics problem}
Now we consider a two-dimensional linear elastodynamics problem. For $T>0$, find $\mathbf{u}\colon\Omega\times [0, T]\to\mathbb{R}^m$ such that 
$$\begin{cases}
\rho \ddot{ \mathbf{u}}+2\rho\gamma\dot{\mathbf{u}}+\rho\gamma^2\mathbf{u}-\nabla \cdot \bm{\sigma}=\mathbf{f} &\quad \text{ in } \Omega\times (0,T],\\

\mathbf{u}=0 &\quad\text{on } \partial\Omega\times (0,T], \\
\mathbf{u}(x,0)=\mathbf{u}_0(x), \quad \dot{\mathbf{u}}(x,0)=\mathbf{u}_1(x) &\quad \text{ in }\Omega.
\end{cases}$$
Here $\mathbf{f}\in L^2((0,T], L^2(\Omega))$ is the source term, and $\rho\in L^{\infty}(\Omega)$ is such that $\rho=\rho(\mathbf{x})>0$ for almost any $\mathbf{x}\in\Omega$. The stress tensor $\bm{\sigma}$ is defined through Hooke's law, that is, 
\begin{equation}
\bm{\sigma} =2\mu \bm{\varepsilon} +\lambda \mathrm{tr}(\bm{\varepsilon})\mathrm{Id},
\end{equation}
where $\mathrm{Id}$ is the identity matrix, $\mathrm{tr}$ is the trace operator, and 
\begin{equation}
\bm{\varepsilon}(\mathbf{u})=\frac{1}{2}(\nabla \mathbf{u}+\nabla \mathbf{u}^T).
\end{equation}
We consider $\Omega=(0,1)\times (0,1)$, $T=1$, and set $\rho=1, \gamma=1, \lambda=1$ and $\mu=1$.  Here we choose $\mathbf{u}_0$, $\mathbf{u}_1$ and $\mathbf{f}$ such that the exact solution is 
$$\mathbf{u}(x,t)=\sin(\sqrt{2}\pi t)\begin{bmatrix}
-\sin^2(\pi x)\sin(2\pi y)\\
\sin(2\pi x)\sin^2(\pi y)
\end{bmatrix}.$$  
Then $\mathbf{u}_0(x)= \begin{bmatrix}0  \\ 0\end{bmatrix}$, $\mathbf{u}_1(x)= \sqrt{2}\pi \begin{bmatrix}-\sin^2(\pi x)\sin(2\pi y) \\ \sin(2\pi x)\sin^2(\pi y)\end{bmatrix}$ and 
\begin{align*}
\mathbf{f}(x,t)= &\left[ (6\pi^2 +0.01)
\sin(\sqrt{2}\pi t)+0.2\sqrt{2}\pi\cos(\sqrt{2}\pi t)\right]\begin{bmatrix}-\sin^2(\pi x)\sin(2\pi y) \\ \sin(2\pi x)\sin^2(\pi y)\end{bmatrix}\\
&+2\pi^2\sin(\sqrt{2}\pi t)\begin{bmatrix}\sin(2\pi y)\\ -\sin(2\pi x)\end{bmatrix}.
\end{align*}
We discretize the linearized elastodynamics equations in the same manner as in Section \ref{linear_wave_experiments}. Here we use $CG$--$r$ elements (continuous piecewise polynomials of degree $r$) in space with $r=q$, $h=k$ and compute the errors $|||\mathbf{Z}-\mathbf{Z}_{\mathrm{DG}}|||$ and $\|\dot{\mathbf{u}}(T)-\dot{\mathbf{u}}_{\mathrm{DG}}(t_N^{-})\|_{L^2}+\| \mathbf{u}(T)-\mathbf{u}_{\mathrm{DG}}(t_N^{-})\|_{L^2}$ versus $k$ for $k=0.250, 0.200,0.125$ and $0.100$ with respect to polynomial degrees $q=2,3,4$.  Note that here we use $k=h=0.100$ instead of $0.0625$ for the smallest time step; this is due to the computational limits of FEniCs for the high order approximation of non-scalar problems. In particular, it may take hours to compute the solution when we use $r=q=4$ with $17$ grid points (e.g. $h=0.0625$) on each direction of the chosen $2$-dimensional domain $\Omega=(0,1)\times (0,1)$. Thus, we choose the last step to be $0.100$ to ensure that we still have sufficient data to compute the convergence rates. Both the semi-discrete errors in the energy norm and the fully discrete errors in the $L^2$-norm are computed in Table \ref{tab:table3} and plotted in a log-log scale in 
Fig. \ref{fig:plot3} and Fig. \ref{fig:plot4}.
\begin{figure}[h!]
 \centering
\begin{minipage}[b]{0.7\textwidth}
\includegraphics[width=\textwidth]{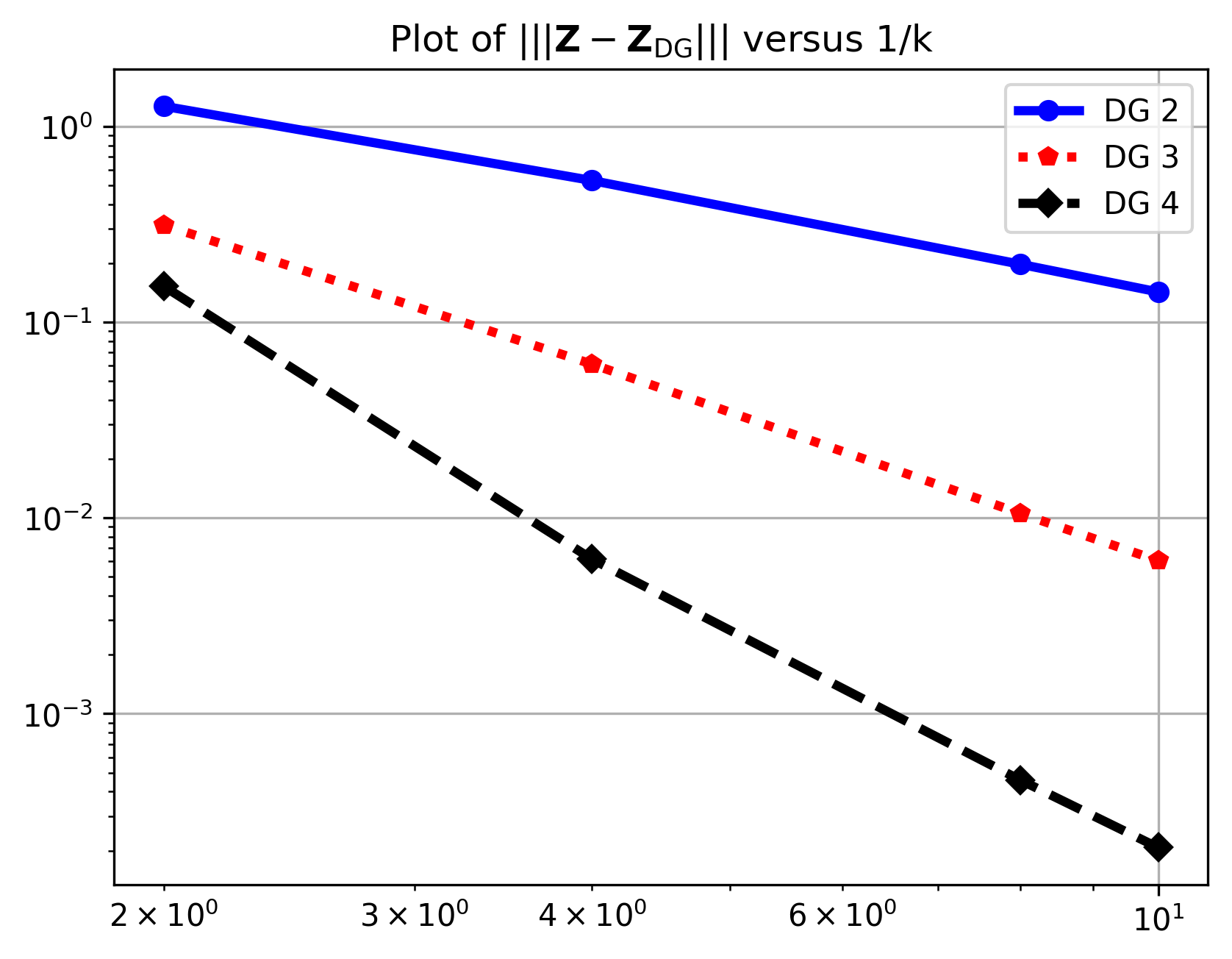}
\caption{$|||\mathbf{Z}-\mathbf{Z}_{\mathrm{DG}}|||$ versus $1/k$, with $k=0.500, 0.250, 0.125$ and $0.100$ in a log-log scale.}\label{fig:plot3}
 \end{minipage}
\hfill
\begin{minipage}[b]{0.7\textwidth}
    \includegraphics[width=\textwidth]{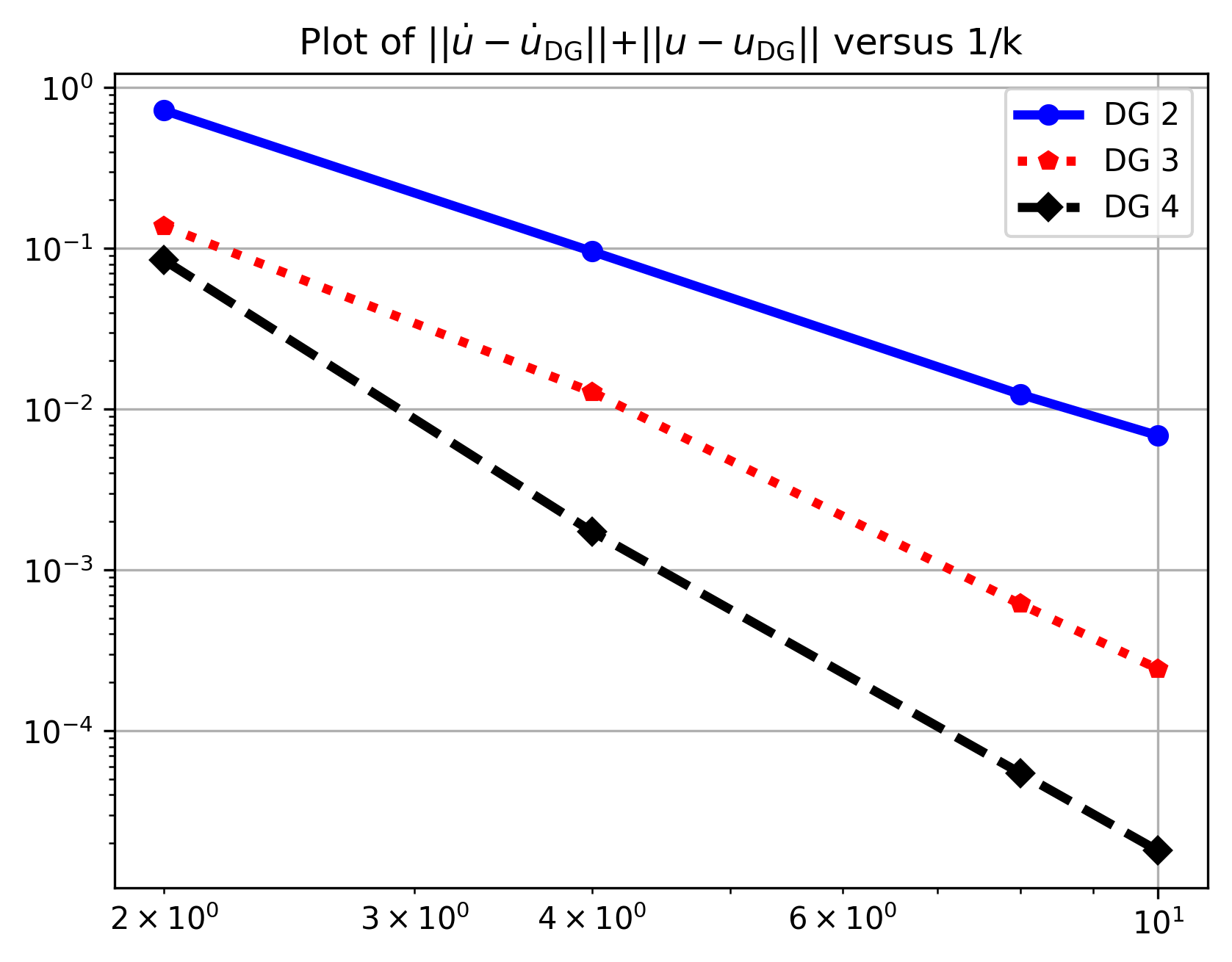}
\caption{$\|\dot{\mathbf{u}}-\dot{\mathbf{u}}_{\mathrm{DG}}\|_{L^2}+\|\mathbf{u}-\mathbf{u}_{\mathrm{DG}}\|_{L^2}$ versus $1/k$, with $k=0.500, 0.250, 0.125$ and $0.100$ in a log-log scale.}\label{fig:plot4}
  \end{minipage}
\end{figure}
  \begin{table}[t]
  \begin{center}
    \caption{Computed errors $|||Z-Z_{\mathrm{DG}}|||$,$||\dot{\mathbf{u}}(T)-\dot{\mathbf{u}}_{\mathrm{DG}}(t_N^{-})||_{L^2} +||\mathbf{u}(T)-\mathbf{u}_{\mathrm{DG}}(t_N^{-})||_{L^2}$ and corresponding convergence rates with respect to polynomial degrees $q=2, 3, 4$.}
    \label{tab:table3}
    \begin{tabular}{c c c c c c } 
     \hline
      $ q$  &    $k$   &  energy-norm error  & rate  & $L^2$-error & rate \\
      \hline 
 $2$  & $5.00\mathrm{e}-1$    & $ 1.2668\mathrm{e}-0$   & ---  & $7.2172\mathrm{e}-1$ & ---  \\
   & $2.50\mathrm{e}-1$    & $ 5.2863\mathrm{e}-1$   & $1.2609$  & $9.5802\mathrm{e}-2$ & $2.9133$ \\
   & $1.25\mathrm{e}-1$    &  $ 1.9754\mathrm{e}-1$   &   $1.4201$&$1.2390\mathrm{e}-2$ &  $2.9509$\\
&    $1.00\mathrm{e}-1$    &  $1.4262\mathrm{e}-2$      &  $1.4599$ & $6.8663\mathrm{e}-3$ &  $2.6452$\\
$3$    & $5.00\mathrm{e}-1$    & $3.1328\mathrm{e}-1$   & ---  & $1.3788\mathrm{e}-1$ &  ---\\
       & $2.50\mathrm{e}-1$    &  $6.0999\mathrm{e}-2$  &   $2.3606$ & $1.2789\mathrm{e}-2$ &  $3.4304$\\
       & $1.25\mathrm{e}-1$    &  $ 1.0548\mathrm{e}-2$   &   $2.5318$ &$6.1569\mathrm{e}-4$ &  $4.3765$\\
&    $1.00\mathrm{e}-1$    &      $6.0522\mathrm{e}-3$  &  $2.4895$ & $2.4334\mathrm{e}-4$ &  $4.1600$  \\  
$4$  & $5.00\mathrm{e}-1$    & $1.5241\mathrm{e}-1$   & ---& $8.4535\mathrm{e}-2$ &  --- \\
       & $2.50\mathrm{e}-1$    &  $6.1539\mathrm{e}-3$  &   $4.6303$                                                                   & $1.7324\mathrm{e}-3$ &  $5.6087$ \\
      & $1.25\mathrm{e}-1$    &  $4.5846\mathrm{e}-4$   &   $3.7466$ & $5.4731\mathrm{e}-5$ &  $4.9843$\\
&    $1.00\mathrm{e}-1$    &  $ 2.0732\mathrm{e}-4$  &  $3.5565$  &$1.7987\mathrm{e}-5$ &  $4.9868 $\\ 
\hline        
    \end{tabular}
  \end{center}
\end{table}

A convergence rate of $O(k^{q-\frac{1}{2}})$ for the energy norm is observed, in accordance with the theoretical result (cf. Remark \ref{uniform_grid}). For the fully discrete errors in the $L^2$-norm, the numerical experiments show a better convergence rate of $O(k^{q+1})$ rather than $O(k^{q}).$ This is reasonable since we use a higher-order degree of approximation in the spatial direction, e.g. we use $r=q$ instead of $q-1$. The motivation for choosing $r=q$ rather than $q-1$ is to ensure a more accurate approximation of the more complicated stiffness matrix in this example. Though the convergence rate should be dominated by $O(k^q)$ rather than $O(k^{r+1})$, the higher regularity of our exact solution may lead to a slightly better convergence rate.
\section*{Acknowledgments}
The author would like to thank Endre S\"{u}li for his helpful discussions and suggestions.

\end{document}